\documentclass[12pt,twoside,a4paper]{amsart}

\usepackage[utf8]{inputenc}
\usepackage{amssymb, amsthm, amsmath}
\usepackage{amsxtra}
\usepackage{enumerate}

\usepackage{leading}
\leading{18pt}

\usepackage[polish, english]{babel}

\usepackage[colorlinks,citecolor=blue,urlcolor=blue]{hyperref}
\usepackage{color}

\hypersetup{
pdfpagemode=UseNone,
pdfdisplaydoctitle=true,
pdfborder={0 0 0}, % No link borders
pdftitle={Periodic perturbations of unbounded Jacobi matrices I: Asymptotics of generalized eigenvectors},
pdfauthor={Grzegorz Świderski and Bartosz Trojan},
pdflang=en-GB
}

\newtheorem{fakt}{Proposition}
\newtheorem{tw}{Theorem}

\newtheorem{hip}{Conjecture}
\newtheorem{corollary}{Corollary}
\newtheorem*{conjecture*}{Conjecture}

\theoremstyle{definition}

\theoremstyle{remark}
\newtheorem{example}{Example}
\newtheorem{remark}{Remark}

\theoremstyle{plain}
\newcounter{thm}

\newtheorem{main_theorem}[thm]{Theorem}

\DeclareMathOperator{\Dom}{Dom}
\DeclareMathOperator{\tr}{tr}
\DeclareMathOperator{\sym}{sym}
\DeclareMathOperator{\discr}{discr}

\newcommand{\norm}[1]{\lVert {#1} \rVert}
\newcommand{\sprod}[2]{\langle {#1}, {#2} \rangle}
\newcommand{\abs}[1]{\lvert {#1} \rvert}
\newcommand{\Id}{\operatorname{Id}}

\newcommand{\NN}{\mathbb{N}}
\newcommand{\sS}{\mathbb{S}}
\newcommand{\RR}{\mathbb{R}}
\newcommand{\ZZ}{\mathbb{Z}}

\newcommand{\calV}{\mathcal{V}}
\newcommand{\calB}{\mathcal{B}}
\newcommand{\calQ}{\mathcal{Q}}
\newcommand{\calC}{\mathcal{C}}
\newcommand{\calF}{\mathcal{F}}

\newcommand{\ue}{\mathrm{e}}

\newcommand{\pl}[1]{\foreignlanguage{polish}{#1}}

\title[Periodic perturbations of unbounded Jacobi matrices I]
{Periodic perturbations of unbounded Jacobi matrices I: Asymptotics of generalized eigenvectors}

\author{Grzegorz Świderski}
\email{gswider@math.uni.wroc.pl}
\address{
	Grzegorz Świderski\\
	Instytut Matematyczny\\
	Uniwersytet Wrocławski\\
	Pl. Grunwaldzki 2/4\\
	50-384 Wrocław\\
	Poland}

\author{Bartosz Trojan}
\email{bartosz.trojan@pwr.edu.pl}
\address{
    \pl{
    Bartosz Trojan\\
    Wydzia\l{} Matematyki,
    Politechnika Wroc\l{}awska\\
    Wyb. Wyspia\'{n}skiego 27\\
    50-370 Wroc\l{}aw\\
    Poland}
}



\begin{document}
\selectlanguage{english}

\begin{abstract}
   We study asymptotics of generalized eigenvectors associated with Jacobi matrices. 
   Under weak conditions on the coefficients we identify when the matrices are self-adjoint
   and we show that they satisfy strong non-subordinacy condition.
\end{abstract}

\maketitle

\section{Introduction}
Jacobi matrix is a matrix defined by two sequences $a = (a_n \colon n \in \mathbb{N})$ and
$b = (b_n \colon n \in \mathbb{N})$ such that $a_n > 0$ and $b_n \in \mathbb{R}$ by the formula
\begin{equation*}
	A =
	\left( 
   	\begin{array}{cccccc}
		b_0 & a_0 & 0   & 0      &\ldots \\
		a_0 & b_1 & a_1 & 0      & \ldots \\
		0   & a_1 & b_2 & a_2    & \ldots \\
		0   & 0   & a_2 & b_3  &  \\
		\vdots & \vdots & \vdots & &  \ddots
	\end{array} 
	\right).
\end{equation*}
The domain of the operator $A$ is $\Dom(A) = \{ x \in \ell^2 \colon A x \in \ell^2\}$, where
\[
	\ell^2 = \Big\{ x \in \mathbb{C}^{\mathbb{N}} \colon \sum_{n=0}^\infty |x_n|^2 < \infty \Big\}.
\]
For a number $\lambda \in \mathbb{R}$, a non-zero sequence $u = (u_n \colon n \in \mathbb{N})$ is called a
\emph{generalized eigenvector} if it satisfies the recurrence relation
\[
   a_{n-1} u_{n-1} + b_n u_n + a_n u_{n+1} = \lambda u_n, \quad (n \geq 1).
\]

The asymptotic behaviour of generalized eigenvectors implies properties of the operator $A$.
In particular, the operator $A$ is self-adjoint if and only if there exists a generalized eigenvector $u$ such that
$u \notin \ell^2$. On the other hand, subordinacy theory developed in \cite{CH1} and \cite{KP1} reduces questions
about absolutely continuous spectrum of $A$ to the asymptotic analysis of generalized eigenvectors. Specifically,
let $I \subset \mathbb{R}$ be an interval and let the operator $A$ be self-adjoint. If for every pair of generalized 
eigenvectors $u$ and $v$ associated with $\lambda \in I$ holds
\[
   \limsup_{n \rightarrow \infty} \frac{\sum_{i=0}^n u_i^2}{\sum_{i=0}^n v_i^2} < \infty,
\]
then the operator $A$ is absolutely continuous on $I$ and $I$ is in the spectrum of $A$.

The aim of this paper is to analyse the asymptotic behaviour of generalized eigenvectors. This type of results
allows us to better understand properties of the operator $A$. In the sequel \cite{GS2}, 
under stronger assumptions and by means of different techniques, we are able to obtain stronger and more constructive 
spectral information about the operator $A$.

For bounded sequences $a$ and $b$ asymptotics of generalized eigenvectors is well understood and precisely 
described, see for example the recent book \cite{BS2}.

In the unbounded case, the problem is more complex. One of the results in this direction was obtained by Clark in
\cite{SLC}, where the author proved that if the sequences $a$ and $b$ satisfy%
\footnote{The total $N$-variation of a sequence $(x_n : n \in \NN)$ is defined by
\[
	\calV_N(x_n : n\in \NN) = \sum_{n = 0}^\infty \abs{x_{n+N} - x_n}.
\]}
\[
	\calV_1\bigg(\frac{1}{a_n} : n\in \NN\bigg) +
	\calV_1\bigg(\frac{a_{n+1}}{a_n} : n \in \NN\bigg) +
	\calV_1(b_n : n\in \NN) < \infty
\]
and 
\[
	\lim_{n \to \infty} \frac{1}{a_n} = 0,
\]
then for every compact interval $I \subset \mathbb{R}$ there are constants $c_1, c_2>0$ such that for every generalized
eigenvector $u$ associated with $\lambda \in I$ one has
\begin{equation}
	\label{eq:4}
	\frac{c_1}{a_n} (u_0^2 + u_1^2) \leq u_{n-1}^2 + u_n^2 \leq \frac{c_2}{a_n} (u_0^2 + u_1^2).
\end{equation}
Notice that, even though $u_n$ may be zero, it cannot happen that both $u_{n-1}$ and $u_n$ are.

In particular, the operator $A$ is self-adjoint if and only if the Carleman's condition is satisfied, i.e.
\begin{equation}
	\label{eq:14}
	\sum_{n = 0}^\infty \frac{1}{a_n} = \infty.
\end{equation}
By means of subordinacy theory, the spectrum of the operator $A$ equals $\RR$ and it is absolutely continuous. 

However, there are very natural cases not covered by \cite{SLC}. For example, sequences of the form
$a_{2k} = a_{2k+1} = \tilde{a}_k$. In this context the sequence $\tilde{a}_n = (n+1)^\kappa$, for $\kappa \in (0, 1]$,
was investigated in \cite{MM1}. While for $\kappa \in (0, 1)$ the spectrum $\sigma(A)$ of the operator $A$ equals $\RR$ and it is
absolutely continuous, for $\kappa = 1$ a new phenomenon occurs. Namely, the operator $A$ is absolutely continuous on
$\RR \setminus [-1/2, 1/2]$ and $\sigma(A) = \RR \setminus (-1/2, 1/2)$, consequently $(-1/2, 1/2)$ is the spectral gap, 
i.e. $\sigma(A) \cap (-1/2, 1/2) = \emptyset$.

The results of Moszyński \cite{MM1} initiated studies to better understand when a spectral gap occurs. 
In \cite{JN1}, the authors considered periodic modulations, i.e.
$a_n = \alpha_n \tilde{a}_n$, $b_n = \beta_n \tilde{b}_n$, where $(\alpha_n : n \in \NN)$ and $(\beta_n : n \in \NN)$
are periodic sequences.
Additive perturbations were investigated under the condition $b_n = 0$. In \cite{JM1}, the periodic perturbations
of $n^\kappa$ for $\kappa \in (0, 1]$ were considered, namely, $a_n = n^\kappa + d_n$, where $(d_n : n \in \NN)$
is a periodic sequence. Non-periodic perturbations were studied in \cite{JD1}. Finally, the case of periodic perturbations
of monotone increasing sequence was investigated in \cite{JD2}. In \cite{JM1, JN1, MM1} the asymptotics of generalized
eigenvectors is also studied.

In this article, we give rather general conditions on the sequences $a$ and $b$, which guarantee that the generalized
eigenvectors have the asymptotics of the form \eqref{eq:4}. In particular, with a help of subordinacy theory, we get results
about the spectrum of the operator $A$. The conditions are flexible enough to cover the cases considered in
\cite{SLC, JD2, JD1, DJMP, JM1, JN1, JNS, MM1}, see Section~\ref{sec:4.3} for details. Moreover, our results implies the
strong non-subordinacy of $A$ and estimates on the density.

The following theorem is a consequence Theorem~\ref{tw:1} and Proposition~\ref{fakt:2} proved in Section~\ref{sec:3.1}.
\begin{main_theorem}
	\label{thm:1}
	Let $N$ be a positive integer. Suppose that
	\[
		\calV_N\bigg(\frac{1}{a_n} : n \in \NN\bigg) +
		\calV_N\bigg(\frac{b_n}{a_n} : n \in \NN\bigg) +
		\calV_1\bigg(\frac{a_{n+N}}{a_n} : n \in \NN\bigg) < \infty.
	\]
	Let
	\begin{enumerate}[(a)]
		\item $\begin{aligned}[b]
			\lim_{n \to \infty} \frac{1}{a_n} = 0;
		\end{aligned}$
		\item $\begin{aligned}[b]
			\lim_{k \to \infty} \frac{b_{kN + j}}{a_{kN+j}} = q_j
		\end{aligned}$
		for $j \in \{0, \ldots, N-1\}$;
		\item $\begin{aligned}[b]
			\lim_{k \to \infty} \frac{a_{kN+j-1}}{a_{kN+j}} = r_j > 0
		\end{aligned}$
		for $j \in \{0, \ldots, N-1\}$,
	\end{enumerate}
	and
	\[
		\calF = \prod_{j = 0}^{N-1} 
		\begin{pmatrix}
			0 & 1 \\
			-r_j & -q_j
		\end{pmatrix}
		\qquad\text{and}\qquad
		E = 
		\begin{pmatrix}
			0 & -1 \\
			1 & 0
		\end{pmatrix}.
	\]
	If \footnote{$X^t$ denotes the transpose of the matrix $X$}
	$\det(E \calF + \calF^t E^t) > 0$, then for every compact interval $I \subset \mathbb{R}$ there are constants
	$c_1, c_2 > 0$ such that for every generalized eigenvector $u$ associated with $\lambda \in I$ one has
   \[
      \frac{c_1}{a_n} (u_0^2 + u_1^2) \leq u_{n-1}^2 + u_n^2 \leq \frac{c_2}{a_n} (u_0^2 + u_1^2).
   \]
\end{main_theorem}
In particular, by taking $N = 1$ in Theorem~\ref{thm:1}, we can obtain a generalization of Clark's result 
(for related results see \cite{JM2, JN2}). Although, Theorem~\ref{thm:1} covers a large class of sequences, it cannot
be applied if $r_0 = r_1 = \cdots = r_{N-1} = 1$ and $q_0 = q_1 = \cdots = q_{N-1} = q$, where
\begin{equation} \label{eq:definicjaQ}
	q \in \bigg\{2 \cos \frac{\pi}{N}, \ldots, 2 \cos\frac{(N-1)\pi}{N} \bigg\}.
\end{equation}
This case, called \emph{critical}, is covered by the next theorem being a consequence of Theorem~\ref{tw:2} and
Proposition~\ref{fakt:3} proven in Section~\ref{sec:3.2}.
\begin{main_theorem}
	\label{thm:2}
	Let $N$ be a positive integer and let $q$ be defined in \eqref{eq:definicjaQ}.
	Suppose that
	\[
		\calV_N\left( a_n - a_{n-1} : n \in \NN \right) + 
		\calV_N\left( \frac{1}{a_n} : n \in \NN\right) +
		\calV_N\big(b_n - q a_n : n \in \NN\big) < \infty.
	\]
	If
	\begin{enumerate}[(a)]
		\item $\begin{aligned}[b]
			\lim_{n \to \infty} \frac{1}{a_n} = 0;
		\end{aligned}$
		\item $\begin{aligned}[b]
			\lim_{n \to \infty} (b_n - q a_n) = 0;
		\end{aligned}$
		\item $\begin{aligned}[b]
			\lim_{k \to \infty} (a_{k N + j} - a_{k N  + j - 1}) = s_j
		\end{aligned}$
		for $j \in \{0, \ldots, N-1\}$,
	\end{enumerate}
   then for every compact interval $I \subset \RR \setminus [\lambda_-, \lambda_+]$ there are constants $c_1>0, c_2>0$ such that for every generalized eigenvector $u$ associated with $\lambda \in I$ one has
   \[
      \frac{c_1}{a_n} (u_0^2 + u_1^2) \leq u_{n-1}^2 + u_n^2 \leq \frac{c_2}{a_n} (u_0^2 + u_1^2).
   \]
	The numbers $\lambda_- \leq \lambda_+$ are the roots of the equation
	\[
 		0 = \lambda^2 \frac{N^2}{4 - q^2} - \lambda \frac{N q S}{4 - q^2}  \\
		+ S \sum_{j=1}^{N-1} s_j v^2_{j-1} - \sum_{i,j=1}^{N-1} s_i s_j v_{i-1} v_{j-1} v_{i-j}
		- \frac{1}{4} S^2,
	\]
	where $S = \lim_{n \rightarrow \infty} (a_{n+N} - a_n)$, $v_n = w_n(q)$ and $w_n$ is the sequence of Chebyshev
	polynomials of the second kind defined in \eqref{eq:1}.
\end{main_theorem}
Theorem \ref{thm:2} provides a new class of sequences such that the operator $A$ potentially has a spectral gap.

In proving Theorem \ref{thm:1} and Theorem \ref{thm:2} we further develop the method used in \cite[Theorem 4.1]{GS1}.
This approach allows us to better understand the reason why there might be a spectral gap. Moreover, our results give
elementary proofs for previously considered sequences.

Let us briefly outline the structure of the article. In Section \ref{sec:2}, we recall basic definitions and introduce
a notion of strong non-subordinacy. Section \ref{sec:3} contains proofs of main theorems. Finally, in Section \ref{sec:4},
we provide applications of main results.

\subsection*{Notation}
We use the convention that $c$ stands for a generic positive constant whose value can change
from line to line. The set of non-negative integers is denoted by $\NN$. 

\section{Preliminaries}
\label{sec:2}
      
For a number $\lambda \in \mathbb{R}$, a non-zero sequence $u = (u_n \colon n \in \mathbb{N})$ is called a
\emph{generalized eigenvector} if 
\begin{equation}
	\label{eq:3}
   	a_{n-1} u_{n-1} + b_n u_n + a_n u_{n+1} = \lambda u_n, \quad (n \geq 1).
\end{equation}
For each $\alpha \in \RR^2 \setminus \{ 0 \}$ there is unique generalized eigenvector $u$ such that
$(u_0, u_1) = \alpha$. If the recurrence relation \eqref{eq:3} holds also for $n = 0$, with the convention that
$u_{-1} = 0$, then $u$ is a \emph{formal eigenvector} of the matrix~$A$ associated with $\lambda$. 

If $A$ is a self-adjoint operator, we define a Borel measure $\mu$ on $\RR$ by setting
\[
    \mu(B) = \big\langle E_A(B) \delta_0, \delta_0\big\rangle,
\]
where $E_A$ is the spectral resolution of $A$ and $\delta_0$ is the sequence with $1$ on $0$th position and $0$ elsewhere.

Following Clark and Hinton \cite{CH1}, we say that a self-adjoint matrix~$A$ satisfies a 
\emph{strong non-subordinacy} condition on a compact interval $I \subset \RR$ if there exists a constant $\beta > 0$
such that for every generalized eigenvectors $u$ and $v$ associated with $\lambda \in I$ and normalized
so that $u_0^2 + u_1^2 = v_0^2 + v_1^2 = 1$ we have
\[
	\limsup_{n \rightarrow \infty} \frac{\sum_{i=0}^n |u_i|^2}{\sum_{i=0}^n |v_i|^2} \leq \beta.
\]
It was proved in \cite{CH1} (see also \cite{SLC}) that a Jacobi matrix $A$ satisfying the strong non-subordinacy
condition on $I$ gives rise to the absolutely continuous measure $\mu$. Moreover, its density is bounded above and
below by $10 \beta/\pi$ and $1/(10\beta \pi)$, respectively.

Given a positive integer $N$, we define the total $N$-variation $\calV_N$ of a sequence of vectors 
$x = \big(x_n : n \in \NN\big)$ from a vector space $V$ by
\[
	\calV_N(x) = \sum_{n = 0}^\infty \norm{x_{n + N} - x_n}.
\]
Observe that if $(x_n : n \in \NN)$ has finite total $N$-variation, then for each $j \in \{0, \ldots, N-1\}$
a subsequence $(x_{k N + j} : k \in \NN)$ is a Cauchy sequence. Total $N$-variation controls supremum of a sequence.
Indeed, we have
\begin{equation}
	\label{eq:28}
	\sup_{n \in \NN} {\norm{x_n}} \leq \calV_N(x_n : n \in \NN) + \max\big\{\norm{x_0}, \ldots, \norm{x_{N-1}}\big\}.
\end{equation}
Moreover, if $V$ is also an algebra, then for any two sequences $(x_n : n \in \NN)$ and $(y_n : n \in \NN)$
we have
\begin{equation}
	\label{eq:29}
	\calV_N(x_n y_n : n \in \NN) \leq \sup_{n \in \NN}{\norm{x_n}}\ \calV_N(y_n : n \in \NN) +
	\sup_{n \in \NN}{\norm{y_n}}\ \calV_N(x_n : n \in \NN).
\end{equation}
For a square matrix $C$, its \emph{symmetrization} is defined by 
\[
	\sym(C) = \frac{C + C^t}{2},
\]
where $C^t$ is the transpose of $C$. With a binary quadratic form, that is a quadratic form on $\RR^2$, represented by
a real symmetric matrix $C$, we associate its \emph{discriminant} given by the formula $\discr(C) = (\tr C)^2 - 4 \det C$.
For a sequence of square matrices $(C_n : n \in \NN)$ and $n_0, n_1 \in \NN$ we set
\[
	\prod_{k=n_0}^{n_1} C_k = 
    \begin{cases} 
    C_{n_1} C_{n_1 - 1} \cdots C_{n_0} & n_1 \geq n_0, \\
    \Id & \text{otherwise.}
    \end{cases}
\]
Finally, let us recall the definition of Chebyshev polynomials of the second kind, i.e. polynomials
$U_n$ satisfying the following relations
\begin{equation}
	\label{eq:20}
	\begin{gathered}
	U_{-1} (x) = 0, \qquad U_0(x) = 1,  \\
	U_{n+1} (x) = 2 x U_n(x) - U_{n-1}, \qquad (n \geq 0).
	\end{gathered}
\end{equation}
There is the explicit formula for $U_n$ valid on the interval $(-1, 1)$. Namely,
\begin{equation}
	\label{eq:2}
	U_n(x) = \frac{\sin\big((n+1) \arccos x \big)}{\sin \arccos x}
\end{equation}
for $x \in (-1, 1)$. For our application, it is more convenient to use polynomials $w_n$ given by
\begin{equation}
	\label{eq:1}
	w_n(x) = U_n(x/2).
\end{equation}

\section{The strong non-subordinacy}
\label{sec:3}
\subsection{Turán-type determinants}
In this section, we study under which conditions the matrix $A$ is self-adjoint and 
satisfies strong non-subordinacy condition.

Let $\Lambda$ be an open subset of $\RR$. Suppose that for each $\lambda \in \Lambda$ there is a sequence
$Q^\lambda = (Q_n^\lambda : n \in \NN)$ of quadratic forms on $\RR^2$. We say that
$\{ Q^\lambda : \lambda \in \Lambda \}$ is \emph{uniformly non-degenerated} on $I$ if there are $c \geq 1$
and $M \geq 1$	such that for all $v \in \RR^2$, $\lambda \in I$ and $n \geq M$
\[
	c^{-1} \norm{v}^2 \leq \abs{Q_n^\lambda(v)} \leq c \norm{v}^2.
\]
We say $\{Q^\lambda : \lambda \in \Lambda\}$ is \emph{almost uniformly non-degenerated} on $\Lambda$ if
it is uniformly non-degenerated on each compact subinterval of $\Lambda$.

Fix a positive integer $N$ and a Jacobi matrix $A$. Let us define a sequence of functions $(S_n : n\in \NN)$,
$S_n : \sS^1 \times \Lambda \rightarrow \RR$ by
\begin{equation}
	\label{eq:6}
	S_n(\alpha, \lambda) = a_{n + N - 1} Q^\lambda_n
	\begin{pmatrix}
			u_{n-1} \\
            u_n
	\end{pmatrix},
\end{equation}
where $u$ is the generalized eigenvector corresponding to $\lambda$ such that $(u_0, u_1) = \alpha \in \sS^1$,
and $\sS^1$ is the unit sphere in $\RR^2$ centred at the origin.
	
In some cases, the sequence $(S_n : n \in \NN)$ is related to \emph{shifted Turán determinants} employed in
\cite{GvA1} (see also \cite{DombrowskiNevai1986}) to study asymptotically periodic Jacobi matrices.
\begin{tw}
	\label{tw:3}
	Let $A$ be a Jacobi matrix and $\{Q^\lambda : \lambda \in I\}$ be a family of sequences of
	binary quadratic forms uniformly non-degenerated on a compact interval $I$. Suppose that there is $c \geq 1$ such
	that for all $\alpha \in \sS^1$ and $\lambda \in I$
	\begin{equation}
		\label{eq:7}
		c^{-1} \leq \abs{S_n(\alpha, \lambda)} \leq c.
	\end{equation}
	Then there is $c \geq 1$ such that for all $\lambda \in I$, and every generalized eigenvector $u$ corresponding to 
	$\lambda$ and $n \geq 1$
	\[
		c^{-1} a_{n+N-1}^{-1} (u_0^2 + u_1^2) \leq u_{n-1}^2+u_n^2 \leq c a_{n+N-1}^{-1} (u_0^2 + u_1^2).
	\]
	In particular, the matrix $A$ is self-adjoint if and only if
	\begin{equation}
		\label{eq:5}
		\sum_{n = 0}^\infty \frac{1}{a_n} = \infty.
	\end{equation}
	Moreover, \eqref{eq:5} implies that the matrix $A$ satisfies the strong non-subordinacy
	condition on~$I$.
\end{tw}
\begin{proof}
	Let $\lambda \in I$ and $u$ be a generalized eigenvector corresponding to $\lambda$ such that
	$(u_0, u_1) = \alpha \in \sS^1$. Since $\{Q^\lambda : \lambda \in I\}$ is uniformly non-degenerated,
	there are $c \geq 1$ and $M \geq 1$ such that for all $n \geq M$
	\[
		c^{-1} a_{n+N-1} (u_{n-1}^2 + u_n^2) 
		\leq 
		\abs{S_n(\alpha, \lambda)} 
		\leq c a_{n+N-1} (u_{n-1}^2 + u_n^2),
	\]
	which together with \eqref{eq:7} implies that there is $c \geq 1$ such that for all $n \geq M$
    \begin{equation}
		\label{eq:8}
        c^{-1} a_{n+N-1}^{-1} \leq u_{n-1}^2 + u_n^2 \leq c a_{n+N-1}^{-1}.
   	\end{equation}
	In particular, $u \notin \ell^2$ if and only if \eqref{eq:5} holds. The last statement combined with 
	\cite[Theorem 2 and Theorem 3]{BS1} is equivalent to self-adjointness of the matrix~$A$.
        
	Finally, by \eqref{eq:8}, we get
	\begin{equation}
		\label{eq:9}
		2^{-1} c^{-1} \sum_{j = M}^n a_{j+N-1}^{-1} 
		\leq
		\sum_{j = M-1} ^n u_j^2 
		\leq 
		c \sum_{j = M}^n a_{j+N-1}^{-1},
	\end{equation}
	where the constant $c \geq 1$ is independent of $\lambda \in I$ and a generalized eigenvector $u$.
	Hence, for any two generalized eigenvectors $u$ and $v$ corresponding to $\lambda \in I$ and
	normalized so that $(u_0, u_1), (v_0, v_1) \in \sS^1$, by \eqref{eq:9} and \eqref{eq:5},
	we have
	\[
		\limsup_{n \to \infty}
		\frac{\sum_{j = 0}^n \abs{u_j}^2}{\sum_{j = 0}^n \abs{v_j}^2}
		\leq
		\limsup_{n \to \infty}
		\frac{\sum_{j = 0}^{M-2} \abs{u_j}^2 + c \sum_{j = M-1}^n a_{n + N - 1}^{-1}}
		{\sum_{j = 0}^{M-2} \abs{v_j}^2 + 2^{-1} c^{-1} \sum_{j = M-1}^n a_{n+N-1}^{-1}}
		\leq
		2 c^2,
	\]
	which finishes the proof.
\end{proof}

\begin{remark}
	Theorem \ref{tw:3} is also interesting in the case when the condition \eqref{eq:5} is not satisfied. 
	Indeed, it gives a criterion to non-self-adjointness of the matrix $A$ and implies asymptotics of the
	generalized eigenvectors. Such information was needed in the proof of \cite[Theorem 4.8 and Theorem 5.9]{BCRSz}.
\end{remark}
     
\begin{corollary}
	\label{cor:1}
	Let $A$ be a Jacobi matrix and $\{Q^\lambda : \lambda \in \Lambda\}$ be a family of sequences of binary
	quadratic forms such that
	\begin{equation}
		\label{eq:48}
		\lim_{n \to \infty} Q_n^\lambda = f(\lambda) \Id
	\end{equation}
	almost uniformly with respect to $\lambda \in \Lambda$, where $f: \Lambda \rightarrow \RR$ is a function without
	zeros. Suppose that for each $j \in \{0, \ldots, N-1\}$
	\begin{equation}
		\label{eq:49}
		\lim_{k \to \infty} \abs{S_{kN+j}(\alpha, \lambda)} = g_j(\alpha, \lambda)
	\end{equation}
	almost uniformly with respect to $\alpha \in \sS^1$ and $\lambda \in \Lambda$.
	Then for any generalized eigenvector $u$ corresponding to $\lambda \in \Lambda$ such that
	$(u_0, u_1) = \alpha \in \sS^1$
	\begin{equation}
		\label{eq:50}
		\lim_{k \to \infty} a_{(k+1)N + j - 1} \big(u_{kN+j-1}^2 + u_{kN + j}^2\big) 
		=
		\frac{g_j(\alpha, \lambda)}{\abs{f(\lambda)}}
	\end{equation}
	almost uniformly with respect to $\alpha \in \sS^1$ and $\lambda \in \Lambda$.
\end{corollary}
\begin{proof}
	Let $\epsilon > 0$ and let $I$ be a compact interval in $\Lambda$. By \eqref{eq:48}, there is $N$ such
	that for all $\lambda \in I$, $n \geq N$ and $v \in \RR^2$
	\[
		(\abs{f(\lambda)} - \epsilon)
		\big(v_1^2 + v_2^2 \big)
		\leq
		\abs{ Q^\lambda_n(v) }
		\leq
		(\abs{f(\lambda)} + \epsilon) \big(v_1^2 + v_2^2\big).
	\]
	Hence, for any generalized eigenvector $u$ corresponding to $\lambda$ with $(u_0, u_1) = \alpha$
	\[
		(\abs{f(\lambda)} - \epsilon) a_{n+N-1} \big(u_{n-1}^2 + u_n^2\big) 
		\leq
		\abs{S_n(\alpha, \lambda)}
		\leq
		(\abs{f(\lambda)} + \epsilon) a_{n+N-1} \big(u_{n-1}^2 + u_n^2\big).
	\]
	By \eqref{eq:49}, for any $j \in \{0, \ldots, N-1\}$
	\[
		\limsup_{k \to \infty} a_{(k+1)N+j-1} \big(u_{kN + j - 1}^2 + u_{kN + j}^2\big) 
		\leq
		\frac{g_j(\alpha, \lambda)}{\abs{f(\lambda)} - \epsilon}
	\]
	and
	\[
		\liminf_{k \to \infty} a_{(k+1)N+j-1} \big(u_{kN + j -1}^2 + u_{kN + j}^2\big)
		\geq
		\frac{g_j(\alpha, \lambda)}{\abs{f(\lambda)} + \epsilon}.
	\]
	Taking $\epsilon$ approaching zero we obtain \eqref{eq:50}.
\end{proof}

\subsection{Regular case}
\label{sec:3.1}
Fix a positive integer $N$ and a Jacobi matrix $A$. For each $\lambda \in \RR$ and $n \in \NN$
we define the \emph{transfer matrix} $B^\lambda_n$ by
\[
	B_n^\lambda = 
	\begin{pmatrix}
	0 & 1 \\
	-\frac{a_{n-1}}{a_n} & \frac{\lambda - b_n}{a_n}
	\end{pmatrix}.
\]
Then for any generalized eigenvector $u$ corresponding to $\lambda$ we have
\[
    \begin{pmatrix}
        u_n \\
        u_{n+1}	
	\end{pmatrix}
    =
    B_n^\lambda
    \begin{pmatrix}
   		u_{n-1}\\
        u_n
    \end{pmatrix}.
\]
Next, let us define a sequence of binary quadratic forms $Q^\lambda$ by the formula
\[
	Q^\lambda_n(v) = \langle E X_n^\lambda v, v \rangle,
\]
where
\begin{equation}
	\label{eq:22}
	X_n^\lambda = \prod_{j = n}^{n+N-1} B_j^\lambda
	\quad\text{and}\quad
    E = 
    \begin{pmatrix}
    	0 & -1 \\
        1 & 0
    \end{pmatrix}.
\end{equation}
The main result of this section is the following theorem.
\begin{tw}
	\label{tw:1}
	Assume that
	\begin{enumerate}[(a)]
		\item 
		$
		\begin{aligned}[t]
			\calV_N\bigg( \frac{1}{a_{n}} : n \in \NN \bigg) 
			+ \calV_N \bigg( \frac{b_n}{a_n} : n \in \NN \bigg)
			+ \calV_1\bigg(\frac{a_{n+N}}{a_n} : n \in \NN\bigg) < \infty;
		\end{aligned}
		$ \label{tw:1:1}
		\item $
		\begin{aligned}[t]
		c_1 < \frac{a_{n+1}}{a_n} < c_2
		\end{aligned}
		$
		for some constants $c_1, c_2 > 0$; \label{tw:1:2}
		\item the family 
		$
		\begin{aligned}
		\big\{Q^\lambda : \lambda \in I \big\}
		\end{aligned}
		$ is uniformly non-degenerated on a compact interval $I$.
	\end{enumerate}
   Then there is $c \geq 1$ such that for all $\lambda \in I$, and every generalized eigenvector $u$ corresponding to 
	$\lambda$ and $n \geq 1$
	\[
		c^{-1} a_n^{-1} (u_0^2 + u_1^2) \leq u_{n-1}^2+u_n^2 \leq c a_n^{-1} (u_0^2 + u_1^2).
	\]
\end{tw}	
\begin{proof}
	Let $S_n$ be a sequence of functions defined by \eqref{eq:6}. In view of
	Theorem \ref{tw:3}, it is enough to show that there is $c \geq 1$ such that
	\begin{equation}
		\label{eq:10}
		c^{-1} \leq \abs{S_n(\alpha, \lambda)} \leq c
	\end{equation}
	for all $\alpha \in \sS^1$ and $\lambda \in I$. The study of the sequence $(S_n : n \in \NN)$ is motivated
	by the method developed in \cite{SLC}.

	Given a generalized eigenvector corresponding to $\lambda \in I$ normalized so that $(u_0, u_1) = \alpha \in \sS^1$,
	we can easily see that for each $n \geq 2$ $u_n$ considered as a function of $\alpha$ and $\lambda$
	is continuous on $\sS^1 \times I$. As a consequence, the function $S_n$ is continuous on
	$\sS^1 \times I$. Since $\{Q^\lambda : \lambda \in I\}$ is uniformly non-degenerated, for each $n \geq M$ the function
	$S_n$ is non-zero and has the same sign for all $\lambda \in I$ and $\alpha \in \sS^1$. Indeed, otherwise there would
	be a non-zero $\alpha \in \sS^1$ and $\lambda \in I$ such that $S_n(\alpha, \lambda) = 0$,
	which would contradict the non-degeneracy of $Q_n^\lambda$. 
	
	Next, we define a sequence of functions $(F_n : n \geq M)$ on $\sS^1 \times I$ by setting
    \[
   		F_n = \frac{S_{n+1} - S_n}{S_n}.
    \]
	Then
    \[
      	\frac{S_n}{S_M} = \prod_{j=M}^{n-1} (1 + F_j).
    \]
	Since each function $F_n$ is continuous, to conclude \eqref{eq:10} it is enough to show
	that the product
	\[
		\prod_{j = M}^n (1 + F_j)
	\]
	converges uniformly on $\sS^1 \times I$. To do so, we are going to prove that the series
	\begin{equation}
		\label{eq:11}
		\sum_{j = M}^\infty \abs{F_j}
	\end{equation}
	converges uniformly on $\sS^1 \times I$. Let $u$ be a generalized eigenvector corresponding to $\lambda \in I$
	such that $(u_0, u_1) = \alpha \in \sS^1$. Using $E^{-1} = -E$, we can write
    \begin{align*}
    	S_{n+1}(\alpha, \lambda) & = a_{n+N}
        \bigg\langle 
        E X_{n+1}^\lambda
        \begin{pmatrix}
        u_{n}\\
        u_{n+1}
        \end{pmatrix}
        ,
        \begin{pmatrix}
        u_n\\
        u_{n+1}
        \end{pmatrix}
        \bigg\rangle\\
        & =
        a_{n+N}
        \bigg\langle 
        E B_{n+N}^\lambda E^{-1} E X_n^\lambda
        \begin{pmatrix}
        u_{n-1}\\
        u_{n}
        \end{pmatrix}
        ,
        B_n^\lambda
        \begin{pmatrix}
        u_{n-1}\\
        u_{n}
		\end{pmatrix}
		\bigg\rangle\\
        & =
		-a_{n+N}
        \bigg\langle 
		(B_n^\lambda)^t	
        E B_{n+N}^\lambda E E X_n^\lambda
		\begin{pmatrix}
        u_{n-1}\\
        u_{n}
		\end{pmatrix}
		,
		\begin{pmatrix}
        u_{n-1}\\
        u_{n}
        \end{pmatrix}
        \bigg\rangle.
    \end{align*}
	Hence, $S_{n+1}(\alpha, \lambda) - S_n(\alpha, \lambda)$ is equal to
	\begin{equation}
        \bigg\langle 
        \Big(
        -a_{n+N}
        (B_n^\lambda)^t	
        E B_{n+N}^\lambda E
        -a_{n+N-1} \Id\Big)
        E X_n^\lambda
        \begin{pmatrix}
        u_{n-1}\\
        u_{n}
        \end{pmatrix}
        ,
        \begin{pmatrix}
        u_{n-1}\\
        u_{n}
        \end{pmatrix}
        \bigg\rangle.
	\end{equation}
	We now compute the matrix $-a_{n+N} (B_n^\lambda)^t E B_{n+N}^\lambda E$
    \begin{align*}
		&
        -a_{n+N}
        \begin{pmatrix}
        0 & -\frac{a_{n-1}}{a_n} \\
        1 & \frac{\lambda-b_n}{a_n}
        \end{pmatrix}
        \begin{pmatrix}
        0 & -1\\
        1 & 0
        \end{pmatrix}
        \begin{pmatrix}
        0 & 1 \\
        -\frac{a_{n+N-1}}{a_{n+N}} & \frac{\lambda - b_{n+N}}{a_{n+N}}
        \end{pmatrix}
        \begin{pmatrix}
        0 & -1\\
        1 & 0
        \end{pmatrix}\\
        &\qquad\qquad\qquad
		=
        \begin{pmatrix}
        -\frac{a_{n-1}}{a_n} & 0\\
        \frac{\lambda - b_n}{a_n} & -1
        \end{pmatrix}
        \begin{pmatrix}
        -a_{n+N} & 0 \\
        -(\lambda - b_{n+N}) & -a_{n+N-1}
        \end{pmatrix} \\
        & \qquad \qquad \qquad
		=
        \begin{pmatrix}
        \frac{a_{n-1}}{a_n} a_{n+N} & 0 \\
        \lambda - b_{n+N} - \frac{\lambda - b_n}{a_n} a_{n+N} & a_{n+N-1}
        \end{pmatrix}.
	\end{align*}
    Therefore,
    \[
    	-a_{n+N} (B_n^\lambda) ^t E B_{n+N}^\lambda E - a_{n+N-1} \Id =
        \begin{pmatrix}
        \frac{a_{n-1}}{a_n} a_{n + N} - a_{n+N-1} & 0 \\
        \lambda - b_{n+N} - \frac{\lambda-b_n}{a_n}a_{n+N} & 0
        \end{pmatrix}.
	\]
   In particular, we can estimate
	\begin{multline}
        \big\lVert
        a_{n+N} (B_n^\lambda)^t E B_{n+N}^\lambda E + a_{n+N-1} \Id
        \big\rVert
        \leq
        a_{n-1}
        \bigg\lvert
        \frac{a_{n+N}}{a_n} - \frac{a_{n+N-1}}{a_{n-1}}
        \bigg\rvert \\
        +
        |\lambda| a_{n+N}
        \bigg\lvert
        \frac{1}{a_{n+N}}
        -
        \frac{1}{a_n}
        \bigg\rvert 
		+
        a_{n+N}
        \bigg\lvert
        \frac{b_n}{a_n} - \frac{b_{n+N}}{a_{n+N}}
        \bigg\lvert.
	\end{multline}
	Finally, since $\{Q^\lambda : \lambda \in I\}$ is uniformly non-degenerated
	\begin{equation}
		\abs{S_n(\alpha, \lambda)} \geq c^{-1} a_{n+N-1} (u_n^2 + u_{n-1}^2)
	\end{equation}
	and
	\[
		\abs{S_{n+1}(\alpha, \lambda) - S_n(\alpha, \lambda)}
		\leq
		c
		\big\lVert a_{n+N} (B_n^\lambda)^t E B_{n+N}^\lambda E + a_{n+N-1} \Id \big\rVert
		(u_n^2 + u_{n-1}^2),
	\]
	one has
	\begin{align*}
		\abs{F_n(\alpha, \lambda)}
       & \leq 
        c^2
        \frac{1}{a_{n+N-1}}
        \big\lVert a_{n+N} (B_n^\lambda)^t E B_{n+N}^\lambda E + a_{n+N-1} \Id \big\rVert \\
        & \leq
        c^2 \bigg(
        \frac{a_{n-1}}{a_{n+N-1}}
        \bigg\lvert
        \frac{a_{n+N}}{a_n} - \frac{a_{n+N-1}}{a_{n-1}}
        \bigg\rvert
        +
        |\lambda|
        \frac{a_{n+N}}{a_{n+N-1}}
        \bigg\lvert
        \frac{1}{a_{n+N}} - \frac{1}{a_n}
        \bigg\rvert \\
        & \phantom{\leq C^2 \bigg(}+
        \frac{a_{n+N}}{a_{n+N-1}}
        \bigg\lvert
        \frac{b_n}{a_n} - \frac{b_{n+N}}{a_{n+N}}
        \bigg\lvert \bigg).
	\end{align*}
	Using \eqref{tw:1:2}, we can estimate 
	\begin{equation}
		\begin{aligned}
		\sum_{n = M}^\infty \abs{F_n} 
		& \leq
		c^2 c_1^{-N} \calV_1\bigg(\frac{a_{n+N}}{a_n} : n \in \NN\bigg)
		+
		c^2 c_2 \calV_N\bigg(\frac{1}{a_n} : n \in \NN\bigg) \\
		& \phantom{\leq}+
		c^2 c_2 \calV_N\bigg(\frac{b_n}{a_n} : n \in \NN\bigg),
		\end{aligned}
	\end{equation}
	thus, \eqref{tw:1:1} implies \eqref{eq:11} and the proof is complete.
\end{proof}
\begin{corollary}
	\label{cor:2}
	Under the hypothesis of Theorem \ref{tw:1}, the sequence of continuous functions 
	$(S_n : n \in \NN)$ converges almost uniformly to the function without zeros. 
\end{corollary}
Theorem \ref{tw:1} depends on the almost uniform non-degeneracy of $\{Q^\lambda : \lambda \in \RR\}$. We are going
to apply it for a Jacobi matrix $A$ such that for each $j \in \{0, \ldots, N-1\}$ the following limits exist
\begin{equation}
	\label{eq:52}
	q_j = \lim_{k \to \infty} \frac{b_{kN + j}}{a_{kN+j}},
	\quad\text{and}\quad
	r_j = \lim_{k \to \infty} \frac{a_{kN+j-1}}{a_{kN+j}},
\end{equation}
with $r_j > 0$. The next proposition is an easy to check criterion for the family $\{Q^\lambda : \lambda \in \RR\}$
to be almost uniformly non-degenerated.
\begin{fakt}
	\label{fakt:2}
	Suppose that
	\[
		\lim_{n \to \infty} a_n = \infty,
	\]
	Let 
	\[
        \calF = \prod_{j=0}^{N-1}
        \begin{pmatrix}
        	0 & 1 \\
            -r_j & -q_j
        \end{pmatrix}
		\quad\text{and}\quad
		E = 
		\begin{pmatrix}
			0 & -1 \\
			1 & 0
		\end{pmatrix},
	\]
	where $q_j$ and $r_j$ are given by \eqref{eq:52}. Then the family
	$\{Q^\lambda : \lambda \in \RR\}$ is almost uniformly non-degenerated whenever
	\begin{equation}
		\label{eq:15}
		\det (\sym (E\calF)) > 0.
	\end{equation}
\end{fakt}
\begin{proof}
	For each $j \in \{0, \ldots, N-1\}$ we set
	\begin{equation}
		\label{eq:16}
		\calB_j = 
		\begin{pmatrix}
        	0 & 1 \\
        	-r_j & -q_j
		\end{pmatrix}
		=
		\lim_{k \to \infty} B_{k N + j}^\lambda.
	\end{equation}
	We notice that the convergence in \eqref{eq:16} is almost uniform with respect to $\lambda$. For each $\lambda \in \RR$
	and $j \in \{0, \ldots, N-1\}$ we set
	\[
		\calQ_j^\lambda(v) = \lim_{k \to \infty} Q_{kN+j}^\lambda(v),
	\]
	where the convergence is almost uniform with respect to $\lambda$. The quadratic form $\calQ^\lambda_j$ corresponds
	to the matrix
	\begin{equation}
		\label{eq:17}
		\sym\Big(E \prod_{k=0}^{j-1} \calB_k \prod_{k = j}^{N-1} \calB_k\Big).
	\end{equation}
	Hence, it is enough to show that the determinant of the matrix \eqref{eq:17} is positive. Since
	for any matrix $X \in M_{2}(\RR)$
	\[
		\det (\sym (E X)) = - 4 \discr X
	\]
	and the discriminant is invariant under conjugation, we conclude
	\begin{align*}
		\det \sym \Big(E \prod_{k=0}^{j-1} \calB_k \prod_{k = j}^{N-1} \calB_k\Big)
		& = 
		-4 \discr \Big(\prod_{k = 0}^{j-1} \calB_k \prod_{k = j}^{N-1} \calB_k\Big) \\
		& =
		-4 \discr \Big(\prod_{k = 0}^{N-1} \calB_k \Big) 
		=
		\det \sym (E \calF). \qedhere
	\end{align*}
\end{proof}

\begin{remark}
	\label{rem:2}
	There are cases when the condition \eqref{eq:15} can be easily verified. For example, when $\det \calF = 1$ and
	$\abs{\tr \calF} < 2$. Indeed, for the proof it is enough to observe that 
	\[
		\det (\sym (E \calF)) = -\frac{1}{4} (\tr \calF - 2)(\tr \calF + 2).
	\]
\end{remark}

\begin{remark}
	\label{rem:3}
	We can obtain an important example by taking $N$ to be a positive odd integer and $q_j = 0$ for all
	$j \in \{0, \ldots, N-1\}$. To show \eqref{eq:15}, we observe that
	\[
		\begin{pmatrix}
			0    & 1 \\
			-r_{j+1} & 0
		\end{pmatrix}
		\begin{pmatrix}
			0  & 1 \\
			-r_j & 0
		\end{pmatrix}
		=
		\begin{pmatrix}
			-r_j & 0 \\
			0 & -r_{j+1}
		\end{pmatrix}.
	\]
	Therefore	
	\[
		E \calF 
		=
		(-1)^{\lfloor N/2 \rfloor}
			\begin{pmatrix}
				r_0 r_2 \ldots r_{N-1} & 0 \\
				0 & r_1 r_3 \ldots r_{N-2}
			\end{pmatrix},
	\]
	which makes clear that the condition \eqref{eq:15} is satisfied.
\end{remark}

\begin{remark}
	\label{rem:1}
	Let 
	\begin{equation}	
		\label{eq:21}
		q \in (-2, 2) \setminus 
		\bigg\{2 \cos \frac{\pi}{N}, 2 \cos \frac{2\pi}{N}, \ldots, 2\cos\frac{(N-1)\pi}{N}\bigg\}.
	\end{equation}
	Suppose that for all $j \in \{0, \ldots, N-1\}$, $r_j = 1$ and $q_j = q$. In this case we set
	\[
		\calB
		= 
		\begin{pmatrix}
			0 & 1 \\
			-1 & -q
		\end{pmatrix}
		=
		\lim_{n \to \infty} B^\lambda_n.
	\]
	It may be easily verified that
	\begin{equation}
		\label{eq:19}
		\calB^n = (-1)^n
		\begin{pmatrix}
			-w_{n-2}(q) & -w_{n-1}(q) \\
			w_{n-1}(q) & w_n(q)
		\end{pmatrix},
	\end{equation}
	where $(w_n : n \in \NN)$ is the sequence of polynomials defined in \eqref{eq:1}. Therefore, by the recurrence
	relation \eqref{eq:20}, we can write
	\[
		\det(\sym(E \calB^N)) = w_{N-1}^2(q) \left( 1 - \frac{1}{4} q^2 \right),
	\]
	which is positive since, by \eqref{eq:2} and \eqref{eq:21}, $w_{N-1}(q) \neq 0$ and $\abs{q} < 2$.
\end{remark}

\subsection{Critical case}
\label{sec:3.2}
Remark \ref{rem:1} leaves out a very interesting case that is when $w_{N-1}(q) = 0$. In fact, we cannot apply
Theorem \ref{tw:1}, since in this case the family $\{Q^\lambda : \lambda \in \RR\}$ is not almost uniformly non-degenerated. To deal
with this case, we need to modify its definition. Our strategy is motivated by \cite[Criterion 2]{MM1}. 

Let us fix
\begin{equation}
	\label{eq:25}
	q = 2 \cos \frac{k_0 \pi}{N},
\end{equation}
for some $k_0 \in \{1, \ldots, N-1\}$. We set
\begin{equation}
	\label{eq:23}
	\gamma = (-1)^N w_N(q) = (-1)^{N+k_0}.
\end{equation}
For each $\lambda \in \RR$ and $n \in \NN$ we define a binary quadratic form
\[
	Q_n^\lambda(v) = a_{n+N-1}\sprod{E X_n^\lambda v}{v},
\]
where $X_n^\lambda$ and $E$ are given by \eqref{eq:22}.

We notice that $\sprod{E v}{v} = 0$. This simple observation lies behind the following non-degeneracy test.
\begin{fakt}
	\label{fakt:1}
	Suppose that there is $\gamma \in \RR$ such that for all $j \in \{0, \ldots, N-1\}$
	\[
		\lim_{k \to \infty} a_{kN + j}(X_{k N + j}^\lambda - \gamma \Id) = \calC_j^\lambda
	\]
	almost uniformly with respect to $\lambda \in \Lambda$. If for all $j \in \{0, \ldots, N-1\}$
	and all $\lambda \in \Lambda$
	\begin{equation}
		\label{eq:24}
		\discr \calC_j^\lambda < 0,
	\end{equation}
	then the family $\{Q^\lambda : \lambda \in \Lambda\}$ is almost uniformly non-degenerated on $\Lambda$.
\end{fakt}
\begin{proof}
	Since for each $j \in \{0, \ldots, N-1\}$
	\[
		Q^\lambda_n(v) = a_{n+N-1} \sprod{E X_n^\lambda v}{v} 
		= \frac{a_{n+N-1}}{a_{n+N}} \big\langle a_{n+N} E (X_n^\lambda - \gamma \Id) v, v \big\rangle,
	\]
	we have
	\[
		\lim_{k \to \infty} Q^\lambda_{kN+j}(v) = \sprod{E \calC_j^\lambda v}{v}
	\]
	almost uniformly with respect to $\lambda \in \Lambda$. In particular, $\Lambda \ni \lambda \mapsto \calC^\lambda_j$
	is a continuous mapping. Therefore, it is enough to check whether
	\[
		\det(\sym (E \calC_j^\lambda)) > 0,
	\]
	which is easily implied by \eqref{eq:24}, because
	\[
		\det (\sym (E \calC_j^\lambda)) = - 4 \discr \calC_j^\lambda > 0.
		\qedhere
	\]
\end{proof}
We start from the simplified version of our problem. 
\begin{fakt}
	\label{fakt:4}
	Suppose that
	\begin{enumerate}[(a)]
		\item $\begin{aligned}[t]
		\lim_{n \rightarrow \infty} a_n = \infty;
		\end{aligned}$ \label{fakt:4:1}
		\item $\begin{aligned}[t]
			\calV_1\left( \frac{1}{a_n} : n \in \NN\right) < \infty.
		\end{aligned}$ \label{fakt:4:2}
	\end{enumerate}
	Given $\lambda \in \RR$, we set
	\[
		\tilde{B}_n^\lambda = 
		\begin{pmatrix}
		0 & 1 \\
		-1 & \frac{\lambda}{a_n} - q
		\end{pmatrix}
	\]
	and
	\[
		\tilde{C}_n^\lambda = a_n \big((\tilde{B}_n^\lambda)^N - \gamma \Id\big),
	\]
	where $q$ and $\gamma$ are defined by \eqref{eq:25} and \eqref{eq:23}, respectively. 
	Then for every compact interval $I$,
	\begin{equation}
		\label{eq:26}
		\sup_{\lambda \in I} \calV_1\big(\tilde{C}_n^\lambda : n \in \NN\big) < \infty.
	\end{equation}
	Moreover,
	\begin{equation}
		\label{eq:27}
		\lim_{n \to \infty} \tilde{C}_n^\lambda =
		(-1)^{N-1} \frac{\lambda N}{4 - q^2} w_N(q)
		\begin{pmatrix}
		q & 2 \\
		-2 & -q
		\end{pmatrix}
	\end{equation}
	almost uniformly with respect to $\lambda$.
\end{fakt}
\begin{proof}
	First, by \eqref{eq:20}, we get
	\[
		(\tilde{B}_n^\lambda)^N = (-1)^N
		\begin{pmatrix}
			-w_{N-2}(x_n) & -w_{N-1}(x_n) \\
			w_{N-1}(x_n) & w_N(x_n)
		\end{pmatrix},
	\]
	where $x_n = q - \lambda/a_n$. Using \eqref{eq:2}, \eqref{eq:1} and \eqref{eq:25}, we can calculate
	\[
		w_{N-2}(q) = (-1)^{k_0+1}, \qquad w_{N-1}(q) = 0, \qquad w_N(q) = (-1)^{k_0}.
	\]
	Therefore, by \eqref{fakt:4:1}, we get
	\[
		\lim_{n \to \infty} (\tilde{B}_n^\lambda)^N = \gamma \Id
	\]
	with $\gamma$ given by \eqref{eq:23}. We now write
	\begin{align*}
		\tilde{C}_n^\lambda
		& =
		(-1)^N
		a_n
		\begin{pmatrix}
			-w_{N-2}(x_n) + w_{N-2}(q) & -w_{N-1}(x_n) + w_{N-1}(q) \\
			w_{N-1}(x_n) - w_{N-1}(q) & w_N(x_n) - w_N(q)
		\end{pmatrix},
	\end{align*}
	which by the mean value theorem implies that the sequence $(\tilde{C}_n^\lambda : n \in \NN)$ converges to the matrix
	\begin{equation}
		\label{eq:43}
		(-1)^N \lambda
		\begin{pmatrix}
			w'_{N-2}(q) & w'_{N-1}(q) \\
			-w'_{N-1}(q) & -w'_N(q)
		\end{pmatrix}.
	\end{equation}
	Moreover, the second application of the mean value theorem gives
	\[
		\bigg|
		\frac{w_{N-2}(q) - w_{N-2}(x_{n+1})}{q - x_{n+1}}
		-
		\frac{w_{N-2}(q) - w_{N-2}(x_n)}{q - x_n}
		\bigg|
		\leq
		c \abs{\lambda} \bigg| \frac{1}{a_{n+1}} - \frac{1}{a_n} \bigg|,
	\]
	for some $c > 0$. Therefore, \eqref{eq:26} follows from \eqref{fakt:4:2}.
	
	It remains to express the matrix \eqref{eq:43} in the form~\eqref{eq:27}. Combining
	\cite[Section~1.2.2, (1.7)]{MH2003} with \cite[Section~2.4.5, (2.48)]{MH2003}, we obtain the following identity
	\[
		w_n' - w_{n-2}' = n w_{n-1},
	\]
	thus
	\[
		w_N'(q) = w_{N-2}'(q).
	\]
	Next, we have (see \cite[Section~2.5, Problem~15]{MH2003})
	\[
		w_n'(x) = \frac{(n+2) w_{n-1}(x) - n w_{n+1}(x)}{4 - x^2},
	\]
	for $x \in (-2, 2)$. Therefore, we get
	\[
		w_{N-1}'(q) = \frac{-2N w_N(q)}{4 - q^2}, \qquad w_N'(q) = \frac{-N w_{N+1}(q)}{4 - q^2}.
	\]
	Finally, by the recurrence formula \eqref{eq:20}, $w_{N+1}(q) = q w_N(q)$. Hence,
	\[
		w_N'(q) = \frac{-N q w_{N}(q)}{4 - q^2},
	\]
	which finished the proof.
\end{proof}
We next investigate for which set $\Lambda \subset \RR$ the family $\{Q^\lambda : \lambda \in \Lambda\}$
is almost uniformly non-degenerated. 
\begin{fakt} 
	\label{fakt:3}
	Let $q$ and $\gamma$ be defined by \eqref{eq:25} and \eqref{eq:23}, respectively. Suppose that
	\begin{enumerate}[(a)]
		\item $\begin{aligned}[b]
		\lim_{n \to \infty} a_n = \infty;
		\end{aligned}$
		\item $\begin{aligned}[b]
		\calV_N\left( a_n - a_{n-1} : n \in \NN \right) + \calV_N\left( \frac{1}{a_n} : n \in \NN\right) < \infty;
		\end{aligned}$ \label{fakt:3:1}
		\item $\begin{aligned}[b]
		\calV_N\big(b_n - q a_n : n \in \NN \big) < \infty;
		\end{aligned}$ \label{fakt:3:2}
		\item $\begin{aligned}[b]
		\lim_{n \to \infty} (b_n - q a_n) = 0;
		\end{aligned}$
		\item $\begin{aligned}[b]
		\lim_{k \to \infty} (a_{k N + j} - a_{k N + j - 1}) = s_j
		\end{aligned}$ for $j \in \{0, \ldots, N-1\}$.
	\end{enumerate}
	Let
	\begin{equation}
		\label{eq:33}
		C_n^\lambda = a_{n + N - 1} \big(X_n^\lambda - \gamma \Id\big).
	\end{equation}
	Then for every compact interval $I$,
	\begin{equation}
		\label{eq:32}
		\sup_{\lambda \in I} {\calV_N(C_n^\lambda : n \in \NN)} < \infty.
	\end{equation}
	Moreover, the family $\big\{Q^\lambda : \lambda \in \RR \setminus [\lambda_-, \lambda_+]\big\}$ is
	almost uniformly non-degenerated, where $\lambda_- \leq \lambda_+$ are the roots of the equation
	\begin{equation}
		\label{eq:42}
 		0 = \lambda^2 \frac{N^2}{4 - q^2} - \lambda \frac{N q S}{4 - q^2}  \\
		+ S \sum_{j=1}^{N-1} s_j u_{j-1}^2 - \sum_{i,j=1}^{N-1} s_i s_j u_{i-1} u_{j-1} u_{i-j}
		- \frac{1}{4} S^2,
	\end{equation}
   where $u_j = w_j(q)$ and
   \[
      S = \lim_{n \rightarrow \infty} (a_{n+N} - a_n) = \sum_{j=0}^{N-1} s_j.
   \]
\end{fakt}
\begin{proof}
	We begin by deriving some consequences of \eqref{fakt:3:1} and \eqref{fakt:3:2}. Let us observe that in view of
	\eqref{eq:28}, the sequence
	$(a_{n+1} - a_n : n \in \NN)$ is bounded. Since
	\[
		\frac{a_{n+1}}{a_{n}} = 1 + (a_{n+1} - a_{n}) \frac{1}{a_n},
	\]
	we have
	\[
		\lim_{n \to \infty} \frac{a_{n+1}}{a_n} = 1.
	\]
	Moreover, by \eqref{eq:29}, we can estimate
	\begin{align}
		\nonumber
		\calV_N\bigg(\frac{a_{n+1}}{a_n} : n \in \NN\bigg) 
		& = \calV_N\bigg(\frac{a_{n+1}}{a_n} - 1: n \in \NN\bigg) \\
		\nonumber
		& = \calV_N\bigg((a_{n + 1} - a_n) \frac{1}{a_n} : n \in \NN\bigg) \\
		\label{eq:34}
		& \leq
		c
		\calV_N\big(a_{n+1} - a_n : n \in \NN \big) + c\calV_N\bigg(\frac{1}{a_n} : n \in \NN\bigg)
	\end{align}
	for some $c > 0$. Similarly, there is $c > 0$ such that
	\begin{equation}
		\label{eq:31}
		\calV_N\bigg(\frac{a_{n+j}}{a_n} : n \in \NN\bigg)
		\leq
		c
		\calV_N\bigg(\frac{a_{n+1}}{a_n} : n \in \NN\bigg)
	\end{equation}
	for each $j \in \{1, \ldots, N-1\}$. Finally, we can estimate
	\begin{align}
		\nonumber
		\calV_N\bigg(\frac{b_n}{a_n} : n \in \NN\bigg)
		& =
		\calV_N\bigg(\frac{b_n}{a_n} - q : n \in \NN\bigg) \\
		\label{eq:35}
		& \leq
		c
		\calV_N\bigg(\frac{1}{a_n} : n \in \NN \bigg)
		+
		c
		\calV_N\big(b_n - q a_n : n \in \NN\big)
	\end{align}
	for some $c > 0$.

	We now proceed to the proof of \eqref{eq:32}. Fix a compact interval $I$. Let
	\[
		\tilde{B}_n^\lambda = 
		\begin{pmatrix}
			0 & 1 \\
			-1 & \frac{\lambda}{a_n} - q
		\end{pmatrix},
	\]
	and $\tilde{C}^\lambda_n = a_n \big((\tilde{B}^\lambda_n)^N - \gamma \Id\big)$. Let us observe that
	\begin{align*}
		X_n^\lambda = B_{n+N-1}^\lambda B_{n + N - 2}^\lambda \cdots B_n^\lambda
		& =
		\big(B_{n+N-1}^\lambda - \tilde{B}_n^\lambda\big) B_{n+N-2}^\lambda \cdots B_n^\lambda \\
		& \phantom{=} 
		+ \tilde{B}_n^\lambda \big(B_{n+N-2}^\lambda - \tilde{B}_n^\lambda\big) \cdots B_n^\lambda\\
		& \phantom{=}
		+ \cdots + (\tilde{B}_n^\lambda)^{N-1} \big(B_n^\lambda - \tilde{B}_n^\lambda\big) \\
		& \phantom{=}
		+ (\tilde{B}_n^\lambda)^N.
	\end{align*}
	Hence,
	\begin{equation}
		\label{eq:37}
		\begin{aligned}
		C_n^\lambda
		& = a_{n + N - 1} \big(B_{n+N-1}^\lambda - \tilde{B}_n^\lambda\big)
		B_{n+N-2}^\lambda \cdots B_n^\lambda \\
		& \phantom{=}+ \tilde{B}_n^\lambda 
		a_{n+N-1} \big(B_{n+N-2}^\lambda - \tilde{B}_n^\lambda\big) \cdots B_n^\lambda \\
		& \phantom{=}+ \cdots +
		(\tilde{B}_n^\lambda)^{N-1} a_{n+N-1} \big(B_n^\lambda - \tilde{B}_n^\lambda\big) \\
		& \phantom{=} + \frac{a_{n+N-1}}{a_n}\tilde{C}^\lambda_n.
		\end{aligned}
	\end{equation}
	As a consequence of \eqref{eq:29}, to estimate $\calV_N(C_n^\lambda : n \in \NN)$, it is enough to show that each factor
	in \eqref{eq:37} has the supremum and the total $N$-variation bounded uniformly with respect to $\lambda \in I$.
	Since each factor as a function of $\lambda$ is continuous, in view of \eqref{eq:28}, it is enough to prove that
	there is $c > 0$ such that for all $\lambda \in I$
	\begin{gather}
		\label{eq:36a}
		\calV_N\big(B_n^\lambda : n \in \NN) \leq c, \\
		\label{eq:36b}
		\calV_N\big(\tilde{B}_n^\lambda : n \in \NN) \leq c,\\
		\label{eq:36c}
		\calV_N\bigg(\frac{a_{n+N-1}}{a_n} \tilde{C}^\lambda_n : n \in \NN\bigg) \leq c
	\end{gather}
	and for all $j \in \{0, \ldots, N-1\}$
	\begin{equation}
		\label{eq:36d}
        \calV_N\Big(a_{n+N-1} \big(B_{n+j}^\lambda - \tilde{B}_n^\lambda\big) : n \in \NN\Big) \leq c.
	\end{equation}
	We have	
	\[
		\big\lVert B_{n+N}^\lambda - B_n^\lambda \big\rVert
		\leq
		\bigg|\frac{a_{n+N-1}}{a_{n+N}} - \frac{a_{n-1}}{a_n} \bigg|
		+
		\abs{\lambda}\bigg|\frac{1}{a_{n+N}} - \frac{1}{a_n} \bigg|
		+
		\bigg|\frac{b_{n+N}}{a_{n+N}} - \frac{b_n}{a_n} \bigg|,
	\]
	thus by \eqref{eq:34}, \eqref{eq:35} and \eqref{fakt:3:1} we get \eqref{eq:36a}. Similarly, for \eqref{eq:36b} we
	obtain
	\[
		\big\lVert \tilde{B}_{n+N}^\lambda - \tilde{B}_n^\lambda \big\rVert
		\leq
		\abs{\lambda} \bigg|\frac{1}{a_{n+N}} - \frac{1}{a_n} \bigg|.
	\]
	The estimate \eqref{eq:36c} is a consequence of Proposition \ref{fakt:3}, \eqref{eq:34} and \eqref{eq:29}.
	Finally, the matrix $a_{n+N-1}\big(B_{n+j}^\lambda - \tilde{B}_n^\lambda\big)$ is equal to
	\begin{equation}
		\label{eq:38}
		\frac{a_{n+N-1}}{a_{n+j}}
		\begin{pmatrix}
			0 & 0 \\
			a_{n+j} - a_{n+j-1} & \lambda \left( 1 - \frac{a_{n+j}}{a_n} \right) + q a_{n+j} - b_{n+j}
		\end{pmatrix},
	\end{equation}
	hence by \eqref{eq:31}, \eqref{fakt:3:1} and \eqref{fakt:3:2} we obtain \eqref{eq:36d}.

	Next, we turn to proving uniform non-degeneracy of the family $\big\{Q^\lambda : \lambda \in I\big\}$. For
	$j \in \{0, \ldots, N-1\}$ we set
	\[
		\calC^\lambda_j = \lim_{k \to \infty} a_{kN+j} \big(X^\lambda_{kN+j} - \gamma \Id\big).
	\]
	First, we show that the convergence is uniform with respect to $\lambda \in I$. To do this, it is enough
	to analyse each factor in \eqref{eq:37} separately. Since
	\[
		\bigg|
		\frac{\lambda - b_n}{a_n} - q
		\bigg|
		\leq
		\frac{1}{a_n} \big(\abs{\lambda} + \big|b_n - q a_n\big|\big),
	\]
	both sequences $(\tilde{B}^\lambda_n : n \in \NN)$ and $(B^\lambda_n : n \in \NN)$ converge to
	\[
		\calB = \begin{pmatrix}
			0 & 1 \\
			-1 & -q
		\end{pmatrix}
	\]
	uniformly with respect to $\lambda \in I$. Moreover, by \eqref{eq:38}, for $i \in \{0, \ldots, N-1\}$ we have
	\[
		\lim_{k \to \infty}
		a_{(k+1) N + i - 1} \big(B^\lambda_{kN + j + i} - \tilde{B}^\lambda_{kN+i}\big)
		=
		\begin{pmatrix}
			0 &  0 \\
			s_{j+i \bmod N} & 0
		\end{pmatrix}
	\]
	uniformly with respect to $\lambda \in I$. As a consequence
	\begin{equation}
		\label{eq:44}
		\calC_j^\lambda = 
		\sum_{i = 0}^{N-1} 
		\calB^{N-1-i}
		\begin{pmatrix}
			0 & 0 \\
			s_{i + j \bmod N} & 0
		\end{pmatrix}
		\calB^i
		+
		\calC,
	\end{equation}
	where, by Proposition \ref{fakt:4},
	\[
		\calC = \lim_{n \to \infty} \tilde{C}^\lambda_n
	\]
	uniformly with respect to $\lambda \in I$. In light of Proposition \ref{fakt:1}, it is enough to show that
	$\discr(\calC^\lambda_j) < 0$. Since $\calB \calC \calB^{-1} = \calC$ and
	$\calB^N = \gamma \Id$, we have
	\begin{equation}
		\label{eq:41}
		\calC^\lambda_{j+1} = \calB \calC^\lambda_j \calB^{-1},
	\end{equation}
	which reduces our task to $j = 0$. Let us observe that by \eqref{eq:19}
	\[
		\calB^{N-1} 
		\begin{pmatrix}
			0 & 0\\
			s_j & 0\\
		\end{pmatrix}
		=
		(-1)^{N-1} w_N(q)
		\begin{pmatrix}
			s_j & 0 \\
			0 & 0
		\end{pmatrix}.
	\]
	Thus, we obtain 
	\[
		\calC^\lambda_0  = 
		\sum_{j=0}^{N-1} \calB^{-j} 
		\calB^{N-1} 
		\begin{pmatrix}
		0 & 0 \\
		s_j & 0
		\end{pmatrix}
		\calB^j + \calC =
		(-1)^{N-1} w_N(q)
		\sum_{j = 0}^{N-1} \calB^{-j} 
		\begin{pmatrix}
			s_j & 0 \\
			0 & 0
		\end{pmatrix}
		\calB^j + \calC.
	\]
	Since, by \eqref{eq:2} and the trigonometric identity for product of sines
	\begin{equation}
		\label{eq:39}
		u_{i-1} u_{j-1} - u_i u_{j-2} = u_{i-j}
	\end{equation}
	we can compute
	\[
		\calB^{-j} = (-1)^j
		\begin{pmatrix}
			u_j & u_{j-1} \\
			-u_{j-1} & -u_{j-2}
		\end{pmatrix}.
	\]
	Therefore, by \eqref{eq:23}, we have
	\[
		\calC^\lambda_0
		=
		-
		\gamma
		\begin{pmatrix}
			s_0 & 0 \\
			0 & 0
		\end{pmatrix}
		-
		\gamma
		\sum_{j = 1}^{N-1}
		s_j 
		\begin{pmatrix}
			-u_j u_{j-2} & -u_j u_{j-1} \\
			u_{j-1} u_{j-2} & u_{j-1}^2
		\end{pmatrix}
		-
		\gamma \frac{\lambda N}{4 - q^2}
		\begin{pmatrix}
			q & 2 \\
			-2 & -q
		\end{pmatrix}.	
	\]
	We now calculate discriminant of $\gamma^{-1} \calC_0^\lambda$. The trace of $\gamma^{-1} \calC_0^\lambda$ is equal to
	\begin{align*}
		\tr (\gamma^{-1} \calC^\lambda_0)
		& =
		-s_0 - \sum_{j = 1}^{N-1} s_j \big(u_{j-1}^2 - u_j u_{j-2}\big),
		 =
		-S
	\end{align*}
	where in the last equality we have used \eqref{eq:39}. We next compute the determinant of
	$\gamma^{-1} \calC_0^\lambda$. Firstly, the terms containing $\lambda^2$ give
	\[
		-\bigg( \frac{\lambda N q}{4 - q^2} \bigg)^2 + \bigg( \frac{2 \lambda N}{4 - q^2} \bigg)^2 
		= \frac{\lambda^2 N^2}{4 - q^2}.
	\]
	Secondly, we find the terms containing $\lambda$
	\[
		\frac{\lambda N}{4 - q^2}
		\Big( -q s_0 + 
		\sum_{j=1}^{N-1} s_j \Big(q \big(u_j u_{j-2} + u_{j-1}^2\big) 
		- 2 u_{j-1} \big(u_{j-2} + u_j \big)\Big) 
		\Big).
	\]
	By the recurrence formula for $w_{j-2}$ we get
	\[
   		u_{j-1} \big(u_{j-2} + u_j\big) = q u_{j-1}^2.
	\]
	We thus obtain
	\[
		\frac{\lambda N q}{4 - q^2} 
		\Big(
		-s_0 + \sum_{j=1}^{N-1} s_j \big(u_j u_{j-2} - u_{j-1}^2\big) 
		\Big) 
		=
		- 
		\frac{\lambda N q}{4 - q^2} S,
	\]
	where in the last equality we again applied \eqref{eq:39}. Lastly, we compute the free term
	\begin{equation} \label{eq:wyrazWolnyDet}
		\sum_{j=1}^{N-1} s_0 s_j u_{j-1}^2 
		+ 
		\sum_{i,j=1}^{N-1} s_i s_j u_i u_{j-1} \big(u_{i-1} u_{j-2} - u_{i-2} u_{j-1}\big).
	\end{equation}
   We have
   \[
      s_0 = S - \sum_{i=1}^{N-1} s_i.
   \]
   Consequently, \eqref{eq:wyrazWolnyDet} equals
   \[
		S \sum_{j=1}^{N-1} s_j u_{j-1}^2 
		+ \sum_{i,j=1}^{N-1} s_i s_j \big(u_i u_{i-1} u_{j-1} u_{j-2} - u_i u_{j-1}^2 u_{i-2} - u_{j-1}^2\big).
   \]
   Identity \eqref{eq:39} implies that
   \[
      -u_{j-1}^2 (1 + u_i u_{i-2}) = - u_{j-1}^2 u_{i-1}^2,
   \]
	which together with \eqref{eq:39} proves that the free term equals
   \[
      S \sum_{j=1}^{N-1} s_j u_{j-1}^2 - \sum_{i,j=1}^{N-1} s_i s_j u_{i-1} u_{j-1} u_{i-j}.
   \]
	Finally, we write
	\begin{align}
      \label{eq:45}
		\discr(\gamma^{-1} \calC^\lambda_0) & = - \lambda^2 \frac{4 N^2}{4 - q^2} + \lambda
		\frac{4 N q}{4 - q^2} S \\ \nonumber
		&\phantom{=} - 4 S \sum_{j=1}^{N-1} s_j u_{j-1}^2 + 4 \sum_{i,j=1}^{N-1} s_i s_j u_{i-1} u_{j-1} u_{i-j} + S^2.
	\end{align}
	Therefore, $\discr(\calC^\lambda_0) < 0$ if and only if $\lambda \notin [\lambda_-, \lambda_+]$, which
	completes the proof.
\end{proof}
      
We now present two examples where one can give formulas for $\lambda_-$ and $\lambda_+$. 
\begin{example}[Multiple weights]
	Let us fix a positive integer $N$ and
	\[
		q = 2 \cos \frac{k_0 \pi}{N}
	\]
	for some $k_0 \in \{1, \ldots, N-1\}$. Suppose $s_0 = s$ and $s_1 = \ldots = s_{N-1} = 0$. Then 
	\[
		\lambda_- = \frac{(q-2) s}{2 N}, \qquad \lambda_+ = \frac{(q+2) s}{2 N}.
	\]
\end{example}
      
\begin{example}[Additive perturbations]
	Let $N = 2K$ and $q = 0$. Then $w_{2 k} (0) = (-1)^k$ and $w_{2k+1}(0) = 0$ for $k \in \mathbb{Z}$. 
   Set $a = s_0 + s_2 + \ldots s_{N-2}$ and $b = s_1 + s_3 + \ldots s_{N-1}$. Therefore, the equation
	\eqref{eq:42} takes the form
	\[
      0 = \lambda^2 N^2 + 4(a+b)b - 4 b^2 - (a+b)^2.
	\]
	This equation can be written as
	\[
		0 = \lambda^2 N^2 -(a-b)^2.
	\]
	Therefore,
	\[
		\lambda_- = - \frac{\abs{a-b}}{N}, \qquad \lambda_+ = \frac{\abs{a-b}}{N}.
	\]
\end{example}

Finally, the following theorem shows the asymptotics of generalized eigenvectors in the critical case.
\begin{tw} 
	\label{tw:2}
	Let $q$ be defined by \eqref{eq:25} and $\lambda_- \leq \lambda_+$ be the roots of the equation \eqref{eq:42}. 
	Suppose that
	\begin{enumerate}[(a)]
		\item $\begin{aligned}[b]
			\lim_{n \to \infty} a_n = \infty;
		\end{aligned}$
		\item $\begin{aligned}[b]
			\calV_N\left( a_n - a_{n-1} : n \in \NN \right) + \calV_N\left( \frac{1}{a_n} : n \in \NN\right) < \infty;
		\end{aligned}$
 		\item $\begin{aligned}[b]
			\calV_N\big(b_n - q a_n : n \in \NN\big) < \infty;
		\end{aligned}$
		\item $\begin{aligned}[b]
			\lim_{n \to \infty} (b_n - q a_n) = 0;
		\end{aligned}$
		\item $\begin{aligned}[b]
			\lim_{k \to \infty} (a_{k N + j} - a_{k N  + j - 1}) = s_j
		\end{aligned}$
		for $j \in \{0, \ldots, N-1\}$.
	\end{enumerate}
   Then for every compact interval $I \subset \RR \setminus [\lambda_-, \lambda_+]$ there is $c \geq 1$ such that for 
   all $\lambda \in I$, and every generalized eigenvector $u$ corresponding to $\lambda$ and all $n \geq 1$
	\[
		c^{-1} a_n^{-1} (u_0^2 + u_1^2) \leq u_{n-1}^2+u_n^2 \leq c a_n^{-1} (u_0^2 + u_1^2).
	\]
\end{tw}
\begin{proof}
	The method of the proof is similar to that of Theorem \ref{tw:1}. We fix a compact interval inside
	$\RR \setminus [\lambda_-, \lambda_+]$. We consider the sequence of functions $(S_n : n \in \NN)$ defined
	by \eqref{eq:6}. By Theorem \ref{tw:3}, it is enough to show that there is $c \geq 1$ such that
	\[
		c^{-1} \leq \abs{S_n(\alpha, \lambda)} \leq c
	\]
	for all $\alpha \in \sS^1$ and $\lambda \in I$. By Proposition \ref{fakt:3}, the family
	$\big\{Q^\lambda : \lambda \in I \big\}$ is uniformly non-degenerated.
	Again, it is enough to show that the series
	\[
		\sum_{n = M}^\infty \abs{F_n}
	\]
	converges uniformly on $\sS^1 \times I$, where $(F_n : n \geq M)$ is the sequence of functions on $\sS^1 \times I$ 
	defined by
	\[
		F_n = \frac{S_{n + N} - S_n}{S_n}.
	\]
	Indeed,
	\[
		\prod_{j = 0}^{k-1} (1 + F_{jN+M}) = \prod_{j = 0}^{k-1} \frac{S_{(j+1) N+M}}{S_{j N+M}} 
		= \frac{S_{kN+M}}{S_M}.
	\]
	Using \eqref{eq:33}, for a generalized eigenvector $u$ corresponding to $\lambda \in I$ with
	$(u_0, u_1) = \alpha \in \sS^1$, we can write
	\begin{align} \nonumber
		S_{n+N}(\alpha, \lambda) & = 
		a_{n+2N-1}^2 \bigg\langle E X_{n+N}^\lambda
		\begin{pmatrix}
			u_{n+N-1} \\
			u_{n+N}
		\end{pmatrix},
		\begin{pmatrix}
			u_{n+N-1} \\
			u_{n+N}
	\end{pmatrix}
		\bigg\rangle \\
		& = a_{n+2N-1}
		\bigg\langle
		E C_{n+N}^\lambda
		\begin{pmatrix}
			u_{n+N-1} \\
			u_{n+N}
		\end{pmatrix},
		\begin{pmatrix}
			u_{n+N-1} \\
			u_{n+N}
		\end{pmatrix}
		\bigg\rangle.
	\end{align}
	Moreover,
	\begin{align*}
   		S_n(\alpha, \lambda) & = a_{n+N-1}^2
		\bigg\langle 
		E 
		\begin{pmatrix}
			u_{n+N-1} \\
			u_{n+N}
		\end{pmatrix},
		(X_n^\lambda)^{-1}
		\begin{pmatrix}
			u_{n+N-1} \\
			u_{n+N}
		\end{pmatrix}
		\bigg\rangle \\
		& =
		a_{n+N-1}^2 
		\bigg\langle
		\big((X_n^\lambda)^{-1}\big)^t E
		\begin{pmatrix}
			u_{n+N-1} \\
			u_{n+N}
		\end{pmatrix},
		\begin{pmatrix}
			u_{n+N-1} \\
			u_{n+N}
		\end{pmatrix}
		\bigg\rangle.
	\end{align*}
	Observe that for every invertible matrix $X \in M_{2}(\mathbb{R})$ the following formula holds true
	\[
      \frac{1}{\det X} E X = (X^{-1})^t E.
	\]
	Therefore,
	\begin{align} \nonumber
		S_n(\alpha, \lambda) & = \frac{a_{n+N-1}^3}{a_{n-1}} 
		\bigg\langle
   		E X_n^\lambda
		\begin{pmatrix}
			u_{n+N-1} \\
			u_{n+N}
		\end{pmatrix},
		\begin{pmatrix}
			u_{n+N-1} \\
			u_{n+N}
		\end{pmatrix}
		\bigg\rangle \\
		& =
		\frac{a_{n+N-1}^2}{a_{n-1}}
		\bigg\langle
		E C_n^\lambda
		\begin{pmatrix}
			u_{n+N-1} \\
			u_{n+N}
		\end{pmatrix},
		\begin{pmatrix}
			u_{n+N-1} \\
			u_{n+N}
		\end{pmatrix}
		\bigg\rangle.
	\end{align}
	In particular, since $\big\{Q^\lambda : \lambda \in I\big\}$ is uniformly non-degenerated, we have
	\begin{equation}
		\abs{S_n(\alpha, \lambda)}
		\geq
		c^{-1}
		\frac{a_{n+N-1}^2}{a_{n-1}} 
		\big(u_{n+N-1}^2 + u_{n+N}^2\big)
	\end{equation}
	for all $n \geq M$, $\alpha \in \sS^1$ and $\lambda \in I$. Hence,
	\begin{equation}
		\big|F_n(\alpha, \lambda) \big|
		\leq
		c 
		\frac{a_{n-1}}{a_{n+N-1}} 
		\bigg\lVert \frac{a_{n+2N-1}}{a_{n+N-1}} C_{n+N}^\lambda - \frac{a_{n+N-1}}{a_{n-1}} C_n^\lambda \bigg\rVert,
	\end{equation}
	which, by \eqref{eq:32}, \eqref{eq:34} and \eqref{eq:29}, is summable uniformly with respect to $\alpha \in \sS^1$
	and $\lambda \in I$.
\end{proof}

\begin{corollary}
	\label{cor:3}
	Under the hypothesis of Theorem \ref{tw:2}, for each $j \in \{0, \ldots, N-1\}$ the subsequence of continuous
	functions $(S_{kN+j} : k \in \NN)$ converges almost uniformly on
	$\sS^1 \times \big(\RR \setminus [\lambda_-, \lambda_+]\big)$ to the function without zeros.
\end{corollary}
      
\begin{remark}
	Theorem~\ref{tw:1} could be proved in the similar way as Theorem~\ref{tw:2}. Then the assumption
	\[
		\calV_1\bigg(\frac{a_{n+N}}{a_n} : n \in \NN\bigg) < \infty
	\]
	has to be replaced by
	\[
		\calV_N\bigg(\frac{a_{n+1}}{a_n} : n \in \NN\bigg) < \infty,
	\]
	which is more symmetric and closer to the approach taken in \cite{MM2}. However, in this way, one could obtain only the
	convergence of a subsequence $(S_{k N + j} : k \in \NN)$ for each $j \in \{0, \ldots, N-1\}$.
\end{remark}

Let $a$ and $b$ be two sequences satisfying hypothesis of Theorem \ref{tw:2}. Observe that if $\tilde{a}$ and
$\tilde{b}$ are two sequences such that for some $d \in \ZZ$ and all $n \in \NN$
\[
	\tilde{a}_{n + d} = a_n, \qquad \tilde{b}_{n + d} = b_n,
\]
then $\tilde{a}$ and $\tilde{b}$ again satisfy hypothesis of Theorem \ref{tw:2}. The resulting sequence 
$(\tilde{s}_0, \ldots, \tilde{s}_{N-1})$ is a cyclic shift of $(s_0, \ldots, s_{N-1})$. Therefore, by
\eqref{eq:41} and \eqref{eq:45}, we easily obtain $\tilde{\lambda}_- = \lambda_-$ and $\tilde{\lambda}_+ = \lambda_+$.

Let us recall that with a Jacobi matrix $A$ one can associate the sequence of polynomials $(p_n : n \geq -1)$
defined by
\begin{align} \label{def:wielomianow}
   \begin{gathered} 
      p_{-1}(\lambda) = 0, \quad p_0(\lambda) = 1, \\
      a_n p_{n+1}(\lambda) = (\lambda - b_n) p_n(\lambda) - a_{n-1} p_{n-1}(\lambda), \quad n \geq 0
   \end{gathered}
\end{align}
for $\lambda \in \RR$.

\begin{fakt} \label{fakt:perturbacjaUn}
	Let $u$ be a generalized eigenvector of $A$ corresponding to $\lambda \in \RR$. Suppose that $u$ is
	not proportional to $(p_n(\lambda) : n \in \NN)$. Then there exist sequences $\tilde{a}$ and $\tilde{b}$ such that
	for all $n \in \NN$
	\begin{equation}
		\label{eq:perturbacjaUn}
        \tilde{a}_{n+2} = a_n, \quad \tilde{b}_{n+2} = b_n,
    \end{equation}
	and
	\[
		u_n = \tilde{p}_{n+2} (\lambda).
	\]
\end{fakt}
\begin{proof}
	Let
	\[
		\beta = (\lambda - b_0) u_0 - a_0 u_1.
    \]
    Because the sequence $u$ is not proportional to $(p_n(\lambda): n \in \NN)$ we have $\beta \neq 0$. In order to
	construct sequences $\tilde{a}$ and $\tilde{b}$ it suffices to solve the system
	\begin{equation}
		\label{eq:46}
		\left\{
		\begin{aligned}
			& \tilde{b}_0 + \tilde{a}_0 \tilde{p}_1(\lambda) = \lambda \\
			& \tilde{a}_0 + \tilde{b}_1 \tilde{p}_1(\lambda) + \tilde{a}_1 u_0 = \lambda \tilde{p}_1(\lambda) \\
			& \tilde{a}_1 \tilde{p}_1(\lambda) + b_0 u_0 + a_0 u_1 = \lambda u_0.
		\end{aligned}
		\right.
	\end{equation}
	Let $\gamma \neq \lambda,\ \gamma \beta > 0$ and
	\[
		1 + \frac{\beta}{\gamma} u_0 > 0.
	\]
	We take
    \[
    	\tilde{b}_0 = \tilde{b}_1 = \lambda - \gamma.
    \]
	Then \eqref{eq:46} may be rewritten as
	\[
		\left\{
		\begin{aligned}
			& \tilde{p}_1(\lambda) = \frac{\gamma}{\tilde{a}_0}\\
			& \gamma \tilde{p}_1(\lambda) = \tilde{a}_0 + \tilde{a}_1 u_0\\
			& \tilde{a}_1 \tilde{p}_1(\lambda) = \beta.
		\end{aligned}
		\right.
	\]
	From the first and the third equation, we get
	\[
		\tilde{a}_1 = \frac{\beta}{\gamma} \tilde{a}_0.
	\]
	Now, from the second equation, we obtain
	\[
		\gamma \tilde{p}_1(\lambda) = \tilde{a}_0 \bigg(1 + \frac{\beta}{\gamma} u_0\bigg).
	\]
	Thus,
	\[
		\tilde{a}_0 = \frac{\abs{\gamma}}{\sqrt{1 + \frac{\beta}{\gamma} u_0}},
	\]
	which completes the proof.
\end{proof}
Given $\alpha \in \sS^1$ and $\lambda \in \RR \setminus [\lambda_-, \lambda_+]$, with a help of Proposition
\ref{fakt:perturbacjaUn}, we may modify the sequences $a$ and $b$ so that 
\[
	S_n(\alpha, \lambda) = \tilde{a}_{n+N+1}^2 
	\big(\tilde{p}_{n+2}(\lambda) \tilde{p}_{n+N+1}(\lambda) - \tilde{p}_{n+1}(\lambda) \tilde{p}_{n+N+2}(\lambda)\big).
\]
Therefore, in order to calculate the limit of $(S_n(\alpha, \lambda) : n \in \NN)$ it is enough to work with
$N$-shifted Tur\'an determinants. This approach is the subject of the forthcoming article \cite{GS2}.
 
\section{Applications} \label{sec:4}

\subsection{Asymptotics of orthonormal polynomials}
   In \cite[Conjecture 6.7]{Ignjatovic2009} the following conjecture has been stated.
   \begin{hip}[Ignjatović \cite{Ignjatovic2009}]
      \label{part3:con:1}
      Assume that $b_n \equiv 0$ and
      \[
         \lim_{n \rightarrow \infty} \frac{a_n}{n^\kappa} = c > 0, \quad \kappa \in (0, 1).
      \]
      Then 
      \[
         \lim_{n \rightarrow \infty} \frac{\sum_{k=0}^n p^2_k(x)}{\sum_{k=0}^n 1/a_k}
      \]
      exists and is positive for $x \in \sigma(A)$, where $(p_n(x) : n \in \mathbb{N})$ is defined in \eqref{def:wielomianow}.
   \end{hip}

   Recently, in \cite{Ignjatovic2014} (see also \cite{IgnjatovicLubinsky2016}) additional conditions were imposed so that the conjecture holds. 
   
   Corollary~\ref{cor:1} and \ref{cor:2} combined with Remark~\ref{rem:3} gives
   \begin{corollary} \label{cor:asymptReg}
      Let $N$ be an odd number. Assume that $b_n \equiv 0$ and
      \[
          \calV_1\bigg(\frac{a_{n+N}}{a_n} : n\in \NN\bigg) + \calV_N\bigg(\frac{1}{a_n} : n\in \NN\bigg) < \infty.
      \]
      If
      \begin{enumerate}[(a)]
         \item
         $\begin{aligned}
            \lim_{n \rightarrow \infty} \frac{a_{n+1}}{a_n} = 1, \quad \lim_{n \rightarrow \infty} a_n = \infty,
         \end{aligned}$
         
         \item
         $\begin{aligned}
            \sum_{n=0}^\infty \frac{1}{a_n} = \infty,
         \end{aligned}$            
      \end{enumerate}
      then
      \[
         \lim_{n \rightarrow \infty} \frac{\sum_{k=0}^n p^2_k(x)}{\sum_{k=0}^n 1/a_k}
      \]
      exists and is positive for $x \in \RR$.
   \end{corollary}

   Corollary~\ref{cor:asymptReg} provides sufficient conditions under which the conjecture holds. Our conditions are actually weaker than in \cite{Ignjatovic2014}. For detailed discussion and comparisons we refer to \cite{GS2}.

\subsection{Classical families of orthogonal polynomials}
   In this section we present three well-known families of orthogonal polynomials which satisfy the assumptions of our theorems. In what follows, we define the density of a measure $\mu$. The corresponding sequence of orthonormal polynomials $(p_n(x) : n \in \NN)$ satisfies the recurrence relation of the form \eqref{def:wielomianow}.
   \begin{enumerate}
      \item[(a)] Generalized Hermite polynomials. Let
         \[
            \mu'(x) = c_t |x|^{t} \ue^{-x^2}
         \]
         for $t > -1$ and a normalizing constant $c_t$. Then 
         \[
            a_n = \frac{1}{\sqrt{2}} \sqrt{n+1+d_n}, \quad d_{2k} = t, \quad d_{2k+1} = 0, \quad b_n \equiv 0
         \]
         (see \cite[page 157]{Chihara1978}). The assumptions of Theorem~\ref{thm:2} are satisfied. If $t \neq 0$, the assumptions of Theorem~\ref{thm:1} are not satisfied.
      
      \item[(b)] Meixner-Pollaczek polynomials. Let
         \[
            \mu'(x) = c_{\lambda, \varphi} \ue^{(2 \varphi - \pi) x} |\Gamma(\lambda + i x)|^2
         \]
         for $\lambda > 0,\ \varphi \in (0, \pi)$ and a normalizing constant $c_{\lambda, \varphi}$. Then
         \[
            a_n = \frac{\sqrt{(n+1) (n + 2 \lambda)}}{2 \sin \varphi}, \quad b_n = \frac{n + \lambda}{\tan \varphi}
         \]
         (see \cite[Chapter 9.7]{koekoek2010hypergeometric}). The assumptions of Theorem~\ref{thm:1} are satisfied.
         
      \item[(c)] Freud polynomials. Let
         \[
            \mu'(x) = c_\beta \ue^{-|x|^\beta}
         \]
         for $\beta > 0$ and a normalizing constant $c_\beta$. Then
         \[
            b_n \equiv 0, \quad \frac{a_n}{(n+1)^{1/\beta}} = c_\beta' + r_n, \quad (r_n : n \in \NN) \in \ell^1
         \]
         for explicit constant $c_\beta'$ (see \cite[Theorem 1.3]{DeiftKriecherbauerMcLaughlinEtAl2001}). The assumptions of Theorem~\ref{thm:1} are satisfied. By Theorem~\ref{tw:3}, the Jacobi matrix associated with sequences $a$ and $b$ is self-adjoint only for $\beta \geq 1$.
   \end{enumerate}
   
   For detailed explanations and numerical tests we refer to \cite{GS2}.

\subsection{Spectral analysis of Jacobi matrices} \label{sec:4.3}
In this section, by applying Theorem \ref{tw:1} and Theorem \ref{tw:2}, we show generalizations of results known in the 
literature. It is worth noting that conclusions from theorems presented in this article are stronger. Namely, we provide 
asymptotics of generalized eigenvectors in different form, and consequently, an additional information about the density.

\subsubsection{Multiple weights}
In \cite{MM1} the author considered Jacobi matrices with $b_n \equiv 0$ and
\[
   a_{2k} = a_{2k+1} = (k+1)^\alpha
\]
for $\alpha \in (0,1]$. The following examples are generalizations of his result.
            
\begin{example}[Regular case]
	Fix a positive integer $N$. Let $(\tilde{a}_n : n \in \NN)$ be a sequence of positive numbers such that
	\begin{enumerate}[(a)]
		\item $\begin{aligned}
			\lim_{n \to \infty} \tilde{a}_n = \infty;
		\end{aligned}$
		\item $\begin{aligned}
			\sum_{n = 0}^\infty \frac{1}{\tilde{a}_n} = \infty;
		\end{aligned}$
		\item $\begin{aligned}
			\calV_1\bigg(\frac{\tilde{a}_{n+1}}{\tilde{a}_n} : n \in \NN\bigg) +
			\calV_1\bigg(\frac{1}{\tilde{a}_n} : n \in \NN\bigg) < \infty.
		\end{aligned}$
	\end{enumerate}
	Let
	\[
		q \in (-2, 2) \setminus \bigg\{ 2 \cos \frac{\pi}{N}, \ldots, 2 \cos \frac{(N-1) \pi}{N} \bigg\}.
	\]
	We set
	\[
		a_{kN} = a_{kN + 1} = \cdots = a_{kN+N-1} = \tilde{a}_k
		\quad\text{and}\quad b_n = q a_n.
	\]
	Then $\sigma(A) = \mathbb{R}$ and the spectrum of the matrix~$A$ is absolutely continuous.
\end{example}

\begin{example}[Critical case]
	\label{ex:1}
	Fix a positive integer $N$. Let $(\tilde{a}_n : n \in \NN)$ be a sequence of positive numbers such that
	\begin{enumerate}[(a)]
		\item $\begin{aligned}
			\lim_{n \to \infty} \tilde{a}_n = \infty;
		\end{aligned}$
		\item $\begin{aligned}
			\calV_1(\tilde{a}_n - \tilde{a}_{n-1} : n \in \NN) + 
			\calV_1\bigg(\frac{1}{\tilde{a}_n} : n \in \NN\bigg) < \infty;
		\end{aligned}$
		\item $\begin{aligned}
			\lim_{n \to \infty} (\tilde{a}_n - \tilde{a}_{n-1}) = s.
		\end{aligned}$
	\end{enumerate}
	Let
	\[
	q = 2 \cos \frac{k_0 \pi}{N}
	\]
	for some $k_0 \in \{1, \ldots, N-1\}$. We set
	\[
		a_{kN} = a_{kN + 1} = \cdots = a_{kN + N-1} = \tilde{a}_k
		\quad\text{and}\quad b_n = q a_n.
	\]
	Then $\sigma(A) \supseteq \RR \setminus (\lambda_-, \lambda_+)$, where 
	\[
		\lambda_- = s \frac{q - 2}{2 N}, \qquad \lambda_+ = s \frac{q + 2}{2N}
	\]
	and the spectrum of the matrix~$A$ is absolutely continuous on $\RR \backslash [\lambda_-, \lambda_+]$.
\end{example}
Whether there is an actual gap around zero depends on a speed of divergence of $(a_n : n \in \NN)$. In particular,
in \cite{DP1}, the authors considered the case when $N$ is an even integer and $q = 0$. They proved that if
\[
	\lim_{n \to \infty} \frac{\tilde{a}_n}{n} = 0,
\]
then zero is not an eigenvalue of $A$ but is an accumulation point of $\sigma(A)$, consequently, there is no gap around zero in the essential spectrum of $A$.
On the other hand, in \cite{HL1} it is proved that if
\[
	\lim_{n \to \infty} \frac{\tilde{a}_n}{n} = \infty,
\]
then the matrix $A$ is self-adjoint and $\sigma(A)$ has no accumulation points.

\subsubsection{Additive periodic perturbations}
In \cite{JM1} and \cite{JNS} the authors studied Jacobi matrices with $b_n \equiv 0$ and
\[
	a_n = (n+1)^\alpha + d_n
\]
for $\alpha \in (0, 1]$, where $(d_n : n \in \NN)$ is a $N$-periodic sequence, i.e. there is the minimal number $N \geq 1$
with the property that $d_{n + N} = d_n$ for all $n \in \NN$. The following examples are generalizations of their results.
\begin{example}[Regular case]
	Fix a positive integer $N$. Let $(\tilde{a}_n : n \in \NN)$ be a sequence of positive numbers such that
	\begin{enumerate}[(a)]
		\item $\begin{aligned}
			\lim_{n \to \infty} \tilde{a}_n = \infty;
		\end{aligned}$
		\item $\begin{aligned}
			\sum_{n = 0}^\infty \frac{1}{\tilde{a}_n} = \infty;
		\end{aligned}$
		\item $\begin{aligned}
			\calV_1\bigg(\frac{\tilde{a}_{n+1}}{\tilde{a}_n} : n \in \NN\bigg)
			+\calV_1\bigg(\frac{1}{\tilde{a}_n} : n \in \NN\bigg) < \infty.
		\end{aligned}$
	\end{enumerate}
	Let
	\[
		q \in (-2, 2) \setminus \bigg\{2 \cos \frac{\pi}{N}, \ldots, 2 \cos \frac{(N-1)\pi}{N} \bigg\},
	\]
	and
	\[
		a_n = \tilde{a}_n + d_n, \quad\text{and}\quad b_n = q a_n.
	\]
	Then $\sigma(A) = \mathbb{R}$ and the spectrum of the matrix~$A$ is absolutely continuous.
\end{example}
            
\begin{example}[$N$ even] \label{example:periodicEven}
	Let $N$ be an even integer. Let $(\tilde{a}_n : n \in \NN)$ be a sequence of positive numbers such that
	\begin{enumerate}[(a)]
		\item $\begin{aligned}
			\lim_{n \to \infty} \tilde{a}_n = \infty;
		\end{aligned}$
		\item $\begin{aligned}
			\calV_1(\tilde{a}_n - \tilde{a}_{n-1} : n \in \NN) + \calV_1\bigg(\frac{1}{\tilde{a}_n} : n \in \NN\bigg) 
			< \infty.
		\end{aligned}$
	\end{enumerate}
	Let
	\[
		a_n = \tilde{a}_n + d_n \quad\text{and}\quad b_n = 0.
	\]
	We set
	\[
		D = \sum_{n = 1}^N (-1)^n d_n.
	\]
	Then $\sigma(A) \supseteq \mathbb{R} \setminus \big(-2 \abs{D}/N, 2 \abs{D}/N\big)$
	and the spectrum of the matrix~$A$ is absolutely continuous on
	$\mathbb{R} \setminus \big[-2 \abs{D}/N, 2\abs{D}/N \big]$.
\end{example}
The result of Example \ref{example:periodicEven} for $N=2$ was proved in \cite[Theorem 5.2]{DJMP}.

Again, whether there is an actual gap around zero requires different techniques.
For $N=2$ results from \cite{DJMP} imply that the set $(-2 \abs{D}/N, 2 \abs{D}/N)$ is a gap in the essential spectrum of $A$, and consequently, the essential spectrum of $A$ equals 
$\mathbb{R} \setminus \big(-2 \abs{D}/N, 2\abs{D}/N \big)$.

\subsubsection{Periodic modulations}
The next example is a periodic modulation considered in \cite{JN1}.

\begin{example}
	Let $N$ be a positive integer. Let $(\tilde{a}_n : n \in \NN)$ and $(\tilde{b}_n : n \in \NN)$ be sequences
	such that $\tilde{a}_n > 0$ and
	\begin{enumerate}[(a)]
		\item $\begin{aligned}
			\lim_{n \to \infty} \tilde{a}_n = \infty;
		\end{aligned}$
		\item $\begin{aligned}
			\sum_{n = 0}^\infty \frac{1}{\tilde{a}_n} = \infty;
		\end{aligned}$
		\item $\begin{aligned}
			\lim_{n \to \infty} \frac{\tilde{b}_n}{\tilde{a}_n} = \delta;
		\end{aligned}$
		\item $\begin{aligned}
			\calV_1\bigg(\frac{\tilde{a}_{n+1}}{\tilde{a}_n} : n \in \NN\bigg) +
			\calV_1\bigg(\frac{1}{\tilde{a}_n} : n \in \NN\bigg) +
			\calV_1\bigg(\frac{\tilde{b}_n}{\tilde{a}_n} : n \in \NN\bigg) < \infty.
		\end{aligned}$
	\end{enumerate}
	Let $(\alpha_n : n \in \NN)$ and $(\beta_n : n \in \NN)$ be $N$-periodic sequences such that the matrix
	\[
		\calF = \prod_{j = 1}^N
		\begin{pmatrix}
			0 & 1 \\
			- \frac{\alpha_{j-1}}{\alpha_j} & -\delta \frac{\beta_j}{\alpha_j}
		\end{pmatrix}
	\]
	satisfies
	\begin{equation}
		\label{eq:54}
		\abs{\tr (\calF)} < 2.
	\end{equation}
	For instance, the condition \eqref{eq:54} holds for $N$ odd and $\delta = 0$. Finally, we set 
	\[
		a_n = \alpha_n \tilde{a}_n\quad\text{and}\quad b_n = \beta_n \tilde{b}_n.
	\]
	Then $\sigma(A) = \mathbb{R}$ and the spectrum of the matrix~$A$ is absolutely continuous.
	According to \cite{JN1}, if $\big| \tr(\calF) \big| > 2$, then the spectrum of $A$ has no accumulation points.
\end{example}

\section*{Acknowledgements}
   The first author would like to thank Ryszard Szwarc for simplification of the proof of Proposition~\ref{fakt:perturbacjaUn}.

\bibliographystyle{abbrv} 
\bibliography{perturbacje}

\end{document}